\documentclass{amsart}

\usepackage{paralist}
\usepackage{color}
\usepackage{mathrsfs}
\usepackage{amssymb}
\usepackage{amsmath}
\usepackage{amsfonts}
\usepackage{stmaryrd}

\newtheorem*{thma}{Theorem~A}
\newtheorem*{thmb}{Theorem~B}
\newtheorem*{thmc}{Theorem~C}
\newtheorem*{thmd}{Theorem~D}
\newtheorem{theorem}{Theorem}[section]
\newtheorem{lemma}[theorem]{Lemma}
\newtheorem{proposition}[theorem]{Proposition}
\newtheorem{corollary}[theorem]{Corollary}
\newtheorem{claim}[theorem]{Claim}

\newtheorem{fact}[theorem]{Fact}

\theoremstyle{definition}
\newtheorem{definition}[theorem]{Definition}
\newtheorem{remark}[theorem]{Remark}

\usepackage[pdftex]{hyperref}
\hypersetup{
    colorlinks=true, 
    linktoc=all,     
    linkcolor=blue,  
    allcolors=blue,
}

\usepackage[framemethod=tikz]{mdframed}

\newcommand{\dom}[1]{\ensuremath{\mathrm{dom}}(#1)}

\newcommand{\power}{\ensuremath{\mathscr{P}}}
\newcommand{\set}[2]{\ensuremath{\{#1 \,|\, #2 \}}}
\newcommand{\seq}[2]{\ensuremath{\langle #1 \,|\, #2 \rangle}}

\newcommand{\rest}[0]{\!\restriction\!}

\newcommand{\el}{\prec}
\newcommand{\iso}{\cong}
\newcommand{\sub}{\subseteq}

\newcommand{\mc}{\mathcal}

\newcommand{\then}{\rightarrow}

\newcommand{\bb}{\mathbb}

\newcommand{\beq}{\begin{equation}}
\newcommand{\eeq}{\end{equation}}
\newcommand{\brm}{\begin{remark}\begin{rm}}
\newcommand{\erm}{\end{rm}\end{remark}}

\newcommand{\bce}{\begin{compactenum}}
\newcommand{\ece}{\end{compactenum}}

\newcommand{\Ult}{\mathrm{Ult}}

\newcommand{\cf}{\mathrm{cf}}

\newcommand{\lev}{\mathrm{lev}}
\newcommand{\acc}{\mathrm{acc}}
\newcommand{\nacc}{\mathrm{nacc}}
\newcommand{\width}{\mathrm{width}}

\newcommand{\Add}{\mathrm{Add}}

\newcommand{\R}{\bb{R}}
\newcommand{\Q}{\bb{Q}}
\renewcommand{\P}{\bb{P}}

\newcommand{\M}{\bb{M}}
\newcommand{\K}{\bb{K}}
\newcommand{\x}{\times}

\newcommand{\TP}{{\sf TP}}
\newcommand{\ITP}{{\sf ITP}}
\newcommand{\SP}{{\sf SP}}
\newcommand{\ISP}{{\sf ISP}}

\newcommand{\ZFC}{\sf ZFC}

\newcommand{\GCH}{\sf GCH}

\newcommand{\PFA}{\sf PFA}
\newcommand{\MA}{\sf MA}

\newcommand{\wKH}{\sf wKH}
\newcommand{\AGP}{\mathsf{AGP}}
\newcommand{\CF}{\mathsf{CF}}
\newcommand{\SCF}{\mathsf{SCF}}
\newcommand{\NS}{\mathsf{NS}}
\newcommand{\SNS}{\mathsf{SNS}}
\newcommand{\wSP}{\mathsf{wSP}}
\newcommand{\wAGP}{\mathsf{wAGP}}
\newcommand{\GMP}{\mathsf{GMP}}

\makeatletter
\@namedef{subjclassname@2020}{\textup{2020} Mathematics Subject Classification}
\makeatother

\begin{document}

\title{Strong tree properties, Kurepa trees, and 
guessing models}

\author{Chris Lambie-Hanson}
\address[Lambie-Hanson]{
Institute of Mathematics, 
Czech Academy of Sciences, 
{\v Z}itn{\'a} 25, Prague 1, 
115 67, Czech Republic
}
\email{lambiehanson@math.cas.cz}
\urladdr{https://users.math.cas.cz/~lambiehanson/}

\author{{\v S}{\'a}rka Stejskalov{\'a}}
\address[Stejskalov{\'a}]{
Charles University, Department of Logic,
Celetn{\' a} 20, Prague~1, 
116 42, Czech Republic
}
\email{sarka.stejskalova@ff.cuni.cz}
\urladdr{logika.ff.cuni.cz/sarka}

\address{Institute of Mathematics, Czech Academy of Sciences, {\v Z}itn{\'a} 25, Prague 1, 115 67, Czech Republic}

\thanks{Both authors were supported by the institutional support RVO:67985840. The first author was supported by GA\v{C}R grant \emph{Compactness in set theory and its applications in algebra and graph theory} (23-04683S). The second author was supported by FWF/GA\v{C}R grant \emph{Compactness principles and combinatorics} (19-29633L).
We would like to thank Rahman Mohammadpour for discussions about an early version of this work and for pointing out the connections between our work and the question of Viale regarding $\omega_2$-guessing models that are not $\omega_1$-guessing models.}

\begin{abstract}
  We investigate the generalized tree properties and guessing model properties introduced by Wei\ss\ and 
  Viale, as well as natural weakenings thereof, studying the relationships among these properties and 
  between these properties and other prominent combinatorial principles. We introduce a weakening of 
  Viale and Wei\ss's Guessing Model Property, which we call the Almost Guessing Property, 
  and prove that it provides an alternate formulation of 
  the slender tree property in the same way that the Guessing Model Property provides and alternate 
  formulation of the ineffable slender tree property. We show that instances of the Almost Guessing 
  Property have sufficient strength to imply, for example, failures of square or the nonexistence of 
  weak Kurepa trees. We show that these instances of the Almost Guessing Property hold in the Mitchell model 
  starting from a strongly compact cardinal and prove a number of other consistency results showing 
  that certain implications between the principles under consideration are in general not reversible. 
  In the process, we provide a new answer to a question of Viale by constructing a model in which, for 
  all regular $\theta \geq \omega_2$, there are stationarily many $\omega_2$-guessing models 
  $M \in \power_{\omega_2} H(\theta)$ that are not $\omega_1$-guessing models.
\end{abstract}

	\maketitle

\tableofcontents

\section{Introduction}

A central line of research in modern combinatorial set theory concerns the study of the extent 
to which various properties of large cardinals can consistently hold at small cardinals, and the 
exploration of the web of implications and non-implications existing among these principles. 
These large cardinal properties can typically be interpreted as assertions of various forms of 
\emph{compactness}, i.e., statements asserting that the global behavior of certain combinatorial 
structures necessarily reflects their local behavior. 

Among the more prominent of these principles in recent years have been various two-cardinal tree 
properties. Their study dates back at least to the 1970s, when they were used to provide 
combinatorial characterizations of strongly compact and supercompact cardinals. In modern 
terminology, Jech \cite{jech_combinatorial_problems} proved that an uncountable cardinal 
$\kappa$ is strongly compact if and only if, for all $\lambda \geq \kappa$, every $(\kappa, 
\lambda)$-list has a cofinal branch. Shortly thereafter, Magidor 
\cite{magidor_combinatorial_characterization} proved that $\kappa$ is supercompact if and only 
if, for all $\lambda \geq \kappa$, every $(\kappa, \lambda)$-list has an ineffable branch 
(see Section \ref{strong_tree_section} for precise definitions of these terms).

The modern study of these two-cardinal tree properties at small cardinals was initiated in the 
2000s by Wei\ss\ \cite{weiss}, who realized that, although it is impossible to have a small cardinal 
$\kappa$ 
for which every $(\kappa, \lambda)$-list has a cofinal branch, if one restricts one's attention to 
certain interesting subclasses of $(\kappa, \lambda)$-lists, one obtains consistent yet nontrivial 
principles. In particular, Wei\ss\ introduced the notions of \emph{thin} and \emph{slender} 
$(\kappa, \lambda)$-lists (with slenderness being a weakening of thinness). These in turn gave rise to 
four compactness principles asserting that, for a given uncountable regular cardinal $\kappa$, for 
every $\lambda \geq \kappa$, every thin (or slender) $(\kappa, \lambda)$-list has a cofinal (or ineffable) 
branch. The strongest of these principles, denoted $\ISP_\kappa$, is the assertion that every 
slender $(\kappa, \lambda)$-list has an ineffable branch. 
Wei\ss\ proved that these principles can consistently hold at small cardinals; for example, 
$\ISP_{\omega_2}$ holds in the supercompact Mitchell model.

In \cite{viale_weiss}, Viale and Wei\ss\ proved that $\ISP_{\omega_2}$ also follows from the 
Proper Forcing Axiom ($\PFA$), in the process giving a useful reformulation of $\ISP_\kappa$ 
in terms of the existence of powerful elementary substructures known as \emph{guessing models} 
(again, see Section \ref{strong_tree_section} for precise definitions). The guessing model 
formulation of $\ISP_\kappa$ has proven quite useful in efficiently establishing that a wide variety 
of other compactness principles are in fact consequences of $\ISP_\kappa$. For example, the following 
is a partial but representative list of some of the consequences of $\ISP_{\omega_2}$:
\begin{itemize}
  \item (Wei\ss\ \cite{weiss}) the failure of $\square(\lambda, \omega_1)$ for every 
  regular $\lambda \geq \omega_2$ (this only requires the weaker $\ITP_{\omega_2}$);
  \item (Viale \cite{viale_guessing_models}, Krueger \cite{krueger_sch}) the Singular Cardinals 
  Hypothesis;
  \item (Lambie-Hanson \cite{narrow_systems} and L\"{u}cke \cite{lucke}, independently) 
  the narrow system property 
  $\mathsf{NSP}(\omega_1, \lambda)$ for every regular $\lambda \geq \omega_2$;
  \item (Cox--Krueger \cite{cox_krueger_indestructible}) the negation of the weak Kurepa Hypothesis 
  ($\neg \wKH$).
\end{itemize}

This paper is the first in a projected series of works further exploring the world of 
these two-cardinal tree properties and their relatives. A future installment 
\cite{arithmetic_paper} will deal primarily with questions connected to cardinal 
arithmetic and cardinal characteristics; the present entry can be seen as an investigation 
into natural weakenings of $\ISP_\kappa$ and the extent to which these weakenings still 
imply various combinatorial consequences of $\ISP_\kappa$.

This investigation inevitably entails a more fine-grained analysis of the notions of 
\emph{slenderness} and \emph{guessing model}. Still leaving the formal definitions for 
Section \ref{strong_tree_section}, let us give a broad overview here. The standard notion 
of ``slender" can naturally be seen as ``$\omega_1$-slender", and the standard notion of 
``guessing model" can naturally be seen as ``$\omega_1$-guessing model". One can obtain 
natural variations of these notions by replacing $\omega_1$ by any other uncountable 
cardinal $\mu$. Following \cite{hachtman_sinapova_super_tree_property}, given 
cardinals $\mu \leq \kappa \leq \lambda$ with $\kappa$ regular, let 
$\SP(\mu, \kappa, \lambda)$ (resp.\ $\ISP(\mu, \kappa, \lambda)$) be the assertion that 
every $\mu$-slender $(\kappa, \lambda)$-list has a cofinal (resp.\ ineffable) branch. 
It will follow immediately from the definitions that these principles become stronger 
as $\mu$ decreases. Arguments of Viale and Weiss \cite{viale_weiss} show that, for 
uncountable cardinals $\mu \leq \kappa$, the following two statements are equivalent:
\begin{itemize}
  \item $\ISP(\mu, \kappa, \lambda)$ holds for every $\lambda \geq \kappa$;
  \item the set of $\mu$-guessing models is stationary in $\power_\kappa H(\theta)$ 
  for every regular $\theta \geq \kappa$.
\end{itemize}
The statement in the second bullet point above is abbreviated here as $\GMP(\mu, \kappa, 
\theta)$ ($\GMP$ stands for ``guessing model property"). In this paper, we isolate a 
weakening of $\GMP$, which we denote by $\AGP$ (for ``almost guessing property"), 
and we also introduce further weakenings of $\SP$ and $\AGP$, $\wSP$ and 
$\wAGP$ respectively, by replacing all instances 
of ``club" and ``stationary" in the definitions by ``strong club" and ``weakly 
stationary", respectively. We then prove the following theorem.

\begin{thma}
  For all uncountable cardinals $\mu \leq \kappa$, with $\kappa$ regular, the following 
  are equivalent:
  \begin{enumerate}
    \item $\SP(\mu, \kappa, \lambda)$ (resp.\ $\wSP(\mu, \kappa, \lambda)$) 
    holds for all $\lambda \geq \kappa$;
    \item $\AGP(\mu, \kappa, \theta)$ (resp.\ $\wAGP(\mu, \kappa, \theta)$) 
    holds for all regular $\theta \geq \kappa$.
  \end{enumerate}
\end{thma}

One reason for introducing these weakenings is the fact that, unlike $\ISP$, they hold 
at cardinals that are strongly compact but not supercompact, or in the model obtained 
by Mitchell forcing starting with a cardinal that is strongly compact but not supercompact. 
For example, we obtain the following theorem, where, for a set $\mc Y \subseteq 
\power_\kappa \lambda$, the principle $\wSP_{\mc Y}(\mu, \kappa, \lambda)$ denotes the 
strengthening of $\wSP(\mu, \kappa, \lambda)$ in which expands the set of $(\kappa, 
\lambda)$-lists under consideration to include those indexed only by elements of 
$\mc Y$ rather than the full $\power_\kappa \lambda$. 

\begin{thmb}
  In the strongly compact Mitchell model, for every cardinal $\lambda \geq \omega_2$, 
  $\wSP_{\mc Y}(\omega_1, \omega_2, \lambda)$ holds, where $\mc Y := \{x \in \power_\kappa 
  \lambda \mid \cf(M \cap \lambda) = \omega_1\}$.
\end{thmb}

While these principles are substantial weakenings of $\ISP$, they are still sufficient 
to obtain some of the combinatorial consequences of $\ISP$. For example, we obtain the 
following results.

\begin{thmc}
  Let $\mu$ be a regular uncountable cardinal.
  \begin{enumerate}
    \item If $\wAGP(\mu, \mu^+, \mu^+)$ holds, then there are no weak 
    $\mu$-Kurepa trees.
    \item Suppose that $\chi < \chi^+ < \kappa \leq \lambda$ are infinite cardinals, 
    with $\kappa$ regular and $\cf(\lambda) \geq \kappa$, and let 
    $\mc Y := \{M \in \power_\kappa H(\lambda^+) \mid 
    \cf(\sup(M \cap \lambda)) > \chi\}$. If $\wAGP_{\mc Y}(\kappa, \kappa, \lambda^+)$ 
    holds, then there are no subadditive, strongly unbounded functions $c:[\lambda]^2 
    \rightarrow \chi$. In particular, $\square(\lambda)$ fails.
  \end{enumerate}
\end{thmc}

We remark that the proofs of Theorems A and B will show that the hypothesis of 
clause (2) of Theorem C holds in the strongly compact Mitchell model with $\chi = \omega$, 
$\mu = \omega_1$, $\kappa = \omega_2$, and any value of $\lambda \geq \omega_2$.

Finally, we prove a consistency result indicating that increasing the first parameter 
in the principles under consideration does lead to strictly weaker statements. In particular, 
we show that, unlike $\ISP_{\omega_2}$ (or even $\wAGP(\omega_1, \omega_2, H(\omega_2))$, 
by clause (1) of Theorem C), the existence of a Kurepa tree is compatible with 
$\ISP(\omega_2, \omega_2, \lambda)$ for all $\lambda \geq \omega_2$.

\begin{thmd}
  Let $\kappa$ be a supercompact cardinal. Then there is a generic extension in which 
  there is an $\omega_1$-Kurepa tree and 
  $\ISP(\omega_2, \omega_2, \lambda)$ holds for every $\lambda \geq \omega_2$. In particular, 
  $\SP(\omega_1, \omega_2, \vert H(\omega_{2})\vert)$ fails.
\end{thmd}

Theorem D can be seen as a variation on a related result of Cummings \cite{C:trees}, 
who proved, using a similar argument, that the tree property at $\omega_2$ is compatible 
with the existence of an $\omega_1$-Kurepa tree. This also addresses a question of 
Viale from \cite{viale_guessing_models} by providing a model in which, for all 
$\theta \geq \omega_2$, there are $\omega_2$-guessing models in $\power_{\omega_2} H(\theta)$ 
that are not $\omega_1$-guessing models.

The structure of the paper is as follows. In Section \ref{generalized_tree_section}, we introduce the 
notion of $\Lambda$-tree for an arbitrary directed partial order $\Lambda$, providing a very general 
setting in which many of our results naturally sit. In Section \ref{two_cardinal_section}, 
we review some basic facts about $\power_\kappa \lambda$ combinatorics. In Section \ref{strong_tree_section}, 
we recall the strong tree properties and guessing model properties introduced in \cite{weiss} and 
\cite{viale_weiss} and establish some basic facts thereon. We also answer a question of 
Fontanella and Matet from \cite{fontanella_matet} by showing that two tree properties considered there 
are equivalent. At the end of the section, we prove a separation result between two strong tree properties 
by showing that one of them entails a failure of approachability while the other does not.

In Section \ref{almost_guessing_section}, we introduce our ``almost guessing properties 
$\mathsf{(w)AGP}(\ldots)$ and prove Theorem A. We also consider a variation of $\GMP(\ldots)$ in which 
only subsets of ``small" sets are guessed and show that this variation does not require the full 
strength of $\ISP(\ldots)$ but already follows from the relevant instance of $\SP(\ldots)$. In 
Section 6, we prove that an instance of $\wAGP(\ldots)$ implies the nonexistence of certain strongly 
unbounded subadditive functions and hence the failure of square, thus establishing clause (2) of Theorem C.
In Section \ref{preservation_section}, we present a variety of preservation lemmas that will be useful 
in our consistency results in the remainder of the paper. In Section \ref{mitchell_section}, we analyze 
various principles in the Mitchell extension obtained from a strongly compact cardinal,
proving Theorem B, among other results. As a corollary, for instance, we obtain a model of 
$\MA + \TP(\omega_2, \geq \omega_2)$ starting just with a strongly compact cardinal. Finally, in 
Section \ref{kurepa_section}, we investigate the effect of various principles on the existence of 
(weak) Kurepa trees, proving clause (1) of Theorem C and Theorem D.

Unless specifically noted, our notation and terminology is standard. We use \cite{jech} as our 
standard background reference for set theory.

\section{Generalized trees} \label{generalized_tree_section}

Although the primary setting for the results of this paper is that of $(\kappa, \lambda)$-
lists, a number of our results hold in a more general setting. We therefore take 
the time here to establish this setting and introduce the notion of generalized trees 
indexed by an arbitrary partial order $\Lambda$. 

Throughout the paper, unless otherwise specified, $\kappa$ denotes an arbitrary regular uncountable 
cardinal. Also, for this section, fix a
partial order $(\Lambda, <_{\Lambda})$. In most cases of interest $\Lambda$ will be $\kappa$-directed, 
but for now let it be arbitrary. We will sometimes abuse notation and simply use $\Lambda$ to 
refer to the order $(\Lambda, <_{\Lambda})$.

\begin{definition}
  A \emph{$\Lambda$-tree} is a pair $T = (\langle T_u \mid u \in \Lambda \rangle, <_T)$ such that the 
  following conditions all hold.
  \bce[(i)]
    \item $\langle T_u \mid u \in \Lambda \rangle$ is a sequence of nonempty, pairwise disjoint sets.
    \item $<_T$ is a transitive partial ordering on $\bigcup_{u \in \Lambda} T_u$.
    \item For all $u,v \in \Lambda$, all $s \in T_u$, and all $t \in T_v$, if 
    $s <_T t$, then $u <_\Lambda v$.
    \item $<_T$ is \emph{tree-like}, i.e., for all $u <_\Lambda v <_\Lambda w$, all 
    $r \in T_u$, all $s \in T_v$ and all $t \in T_w$, if $r, s <_T t$, then $r <_T s$.
    \item For all $u \leq_\Lambda v$ in $\Lambda$ and all $t \in T_v$, there is a unique $s \in T_u$, 
    denoted $t \restriction u$, such that $s \leq_T t$. 
  \ece 
  If $T = (\langle T_u \mid u \in \Lambda \rangle, <_T)$ is a $\Lambda$-tree, we let 
  $\width(T)$ denote the least cardinal $\mu$ such that $\vert T_u\vert  < \mu$ for all $u \in \Lambda$.
  For a cardinal $\kappa$, we say that $T$ is a $\kappa$-$\Lambda$-tree if, in addition to the above requirements, we have 
  $\width(T) \leq \kappa$.
\end{definition}

Given a $\Lambda$-tree $T = (\langle T_u \mid u \in \Lambda \rangle, <_T)$, we will sometimes abuse 
notation and use $T$ to refer to the underlying set of the tree, $\bigcup_{u \in \Lambda} T_u$. 
The set $T_u$ is called the \emph{$u^{\mathrm{th}}$ level} of $T$. If a $\Lambda$-tree $T$ is given, 
it should be understood unless otherwise specified that its $u^{\mathrm{th}}$ level is denoted by 
$T_u$ for each $u \in \Lambda$, and its ordering is denoted by $<_T$. Given $t \in T$, we let 
$\lev_T(t)$ be the unique $u \in \Lambda$ such that $t \in T_u$. A \emph{subtree} of $T$ is a 
$\Lambda$-tree $T' = (\langle T'_u \mid u \in \Lambda \rangle, <_{T'})$ such that 
$T'_u \subseteq T_u$ for all $u \in \Lambda$ and $<_{T'}$ is the restriction of $<_T$ to 
$\bigcup_{u \in \Lambda} T'_u$.

\begin{remark}
   Note that, if $\Lambda = \kappa$ (with the ordinal ordering), then a 
   $\kappa$-$\Lambda$-tree is simply a classical $\kappa$-tree. Another important 
   special case, which will be the setting for many of the results of 
  this paper, is $\Lambda = \power_\kappa \lambda$ for some cardinal $\lambda \geq 
  \kappa$, ordered by $\subseteq$.
\end{remark}

\begin{definition}
  Suppose that $T$ is a $\Lambda$-tree. A \emph{cofinal branch} through $T$ is a function 
  $b \in \prod_{u \in \Lambda} T_u$ such that, for all $u <_\Lambda v$ in $\Lambda$, we have 
  $b(u) <_T b(v)$. The \emph{$(\kappa, \Lambda)$-tree property}, denoted $\TP_\kappa(\Lambda)$, 
  is the assertion that every $\kappa$-$\Lambda$-tree has a cofinal branch.
\end{definition}

\begin{remark}
  Again note that, if $\Lambda = \kappa$ (with the ordinal ordering), then 
  $\TP_\kappa(\Lambda)$ is simply the classical tree property 
  $\TP(\kappa)$.
\end{remark}

Many of the standard facts about $\kappa$-trees remain true about $\kappa$-$\Lambda$-trees when 
$\Lambda$ is $\kappa$-directed. We establish some of these facts now, after some preliminary definitions.

\begin{definition}
  Suppose that $T$ is a $\Lambda$-tree.
  \bce[(i)]
    \item We say that $T$ is \emph{well-pruned} if, for all $u <_\Lambda v$ in $\Lambda$ and all 
    $s \in T_u$, there is $t \in T_v$ such that $s <_T t$.
    \item For a fixed $\mc Y \subseteq \Lambda$, we say that $T$ is 
    \emph{$\kappa$-$\mc Y$-thin} if $\vert T_u\vert  < \kappa$ for all $u \in \mc Y$.
    \item We say that $T$ is \emph{very thin} if $\Lambda$ is $\width(T)^+$-directed.
  \ece
\end{definition}

\begin{proposition} \label{generalized_well_pruned_prop}
  Suppose that $\Lambda$ is $\kappa$-directed, $\mc Y \subseteq \Lambda$ is $<_\Lambda$-cofinal, and 
  $T$ is a $\kappa$-$\mc Y$-thin $\Lambda$-tree. Then $T$ has a well-pruned subtree.
\end{proposition}

\begin{proof}
  For each $u \in \Lambda$, simply let $T'_u$ be the set of all $s \in T_u$ such that, for 
  all $v \in \Lambda$ with $u <_\Lambda v$, there is $t \in T_v$ such that $s <_T t$. Let 
  $<_{T'}$ be the restriction of $<_T$ to $\bigcup_{u \in \Lambda} T'_u$. We claim that 
  $T' = (\langle T'_u \mid u \in \Lambda \rangle, <_{T'})$ is the desired well-pruned 
  subtree.
  
  First note that, for all $u <_\Lambda v$ and all $t \in T'_v$, the definition of $T'$ ensures 
  that $t \restriction u \in T'_u$. We next argue that $T'_u \neq \emptyset$ for all $u \in \Lambda$. 
  Suppose for sake of contradiction that there is $u \in \Lambda$ such that $T'_u = \emptyset$. 
  Note that, if $v \in \Lambda$ and $u <_\lambda v$, then the observation at the beginning of 
  the paragraph implies that we must have $T'_v = \emptyset$ as well. Therefore, by increasing 
  $u$ if necessary, we may assume that $u \in \mc Y$. Since $T'_u = \emptyset$, we know that, for 
  each $s \in T_u$, there is $v_s \in \Lambda$ such that $u <_\Lambda v_s$ and, for all 
  $t \in T_{v_s}$, we have $s \not<_T t$. Since $\Lambda$ is $\kappa$-directed and $\vert T_u\vert  < \kappa$, 
  we can find $v \in \Lambda$ such that $v_s \leq_{\Lambda} v$ for all $s \in T_u$. Choose 
  an arbitrary $t \in T_v$, and let $s := t \restriction u$. Then $t \restriction v_s \in T_{v_s}$ 
  and $s <_T t \restriction v_s$, contradicting our choice of $v_s$. Therefore, $T'_u \neq \emptyset$ 
  for all $u \in \Lambda$. It follows that $T'$ is indeed a subtree of $T$.
  
  It remains to show that $T'$ is well-pruned. Suppose for sake of contradiction that $u <_\Lambda v$ and 
  $s \in T'_u$, but there is no $t \in T'_v$ such that $s <_{T'} t$. Note that, if $v <_\Lambda w$, then 
  there is still no $t \in T'_w$ such that $s <_{T'} t$, as otherwise $t \restriction v$ would 
  contradict the previous sentence. Therefore, by increasing $v$ if necessary, we may assume that 
  $v \in \mc Y$. Let $S:= \{t \in T_v \mid s <_T t\}$. Since $s \in T'_u$, we know that $S$ is nonempty. 
  For each $t \in S$, since $t \notin T'_v$, we can find $w_t \in \Lambda$ such that $v <_\Lambda w_t$ 
  and there is no $r \in T_{w_t}$ such that $t <_T r$. Since $\Lambda$ is $\kappa$-directed and 
  $\vert S\vert  \leq \vert T_v\vert  < \kappa$, we can find $w \in \Lambda$ such that $w_t <_\Lambda w$ for all 
  $t \in S$. Since $s \in T'_u$, we can find $r \in T_w$ such that $s <_T r$. Let $t := r \restriction v$. 
  Then $t \in S$, $r \restriction w_t \in T_{w_t}$, and $t <_T (r \restriction w_t)$, contradicting our 
  choice of $w_t$. Therefore, $T'$ is indeed well-pruned. 
\end{proof}

The following definition will be useful.

\begin{definition}
  Suppose that $T$ is a $\Lambda$-tree and $u <_\Lambda w$. We say that $T$ \emph{splits between 
  $u$ and $v$} if there are distinct $t_0, t_1 \in T_v$ such that $t_0 \restriction u = t_1 \restriction 
  u$.
\end{definition}

Note that, if $T$ is a well-pruned $\Lambda$-tree, $u <_\Lambda v <_\Lambda w$, and 
$T$ splits between $u$ and $v$, then $T$ also splits between $u$ and $w$. The following lemma 
generalizes a result of Kurepa \cite{Kurepa} from the setting of $\kappa$-trees to our general setting.

\begin{lemma}\label{very_thin_branch_lemma}
  Suppose that $T$ is a very thin $\Lambda$-tree. Then $T$ has a cofinal branch.
\end{lemma}

\begin{proof}
  Since $T$ is very thin, by Proposition \ref{generalized_well_pruned_prop} we can assume that it is 
  well-pruned. Let $\mu := \width(T)$; then $\Lambda$ is $\mu^+$-directed and $\vert T_u\vert < \mu$ for 
  all $u \in \Lambda$. We first show that $\mu$ must be a successor cardinal.
  
  \begin{claim}
    There is $\nu < \mu$ such that $\vert T_u\vert  \leq \nu$ for all $u \in \Lambda$.
  \end{claim}  
  
  \begin{proof}
    Let $X := \{\nu \mid \exists u \in \Lambda ~ (\vert T_u\vert  = \nu)\}$. Then $X$ is a set of cardinals less 
    than $\mu$. For each $\nu \in X$, choose $u_\nu \in \Lambda$ such that $\vert T_{u_{\nu}}\vert  = \nu$. 
    Since $\Lambda$ is $\mu^+$-directed, we can find $v \in \Lambda$ such that $u_\nu <_\Lambda v$ 
    for all $\nu \in X$. Since $T$ is well-pruned, it follows $\vert T_v\vert  \geq \vert T_{u_{\nu}}\vert $ for all 
    $\nu \in X$. In particular, $X$ has a maximal element, and $\max(X)$ is as desired in the statement 
    of the claim.
  \end{proof}
  By minimality of $\mu$, it must be the case that $\mu = \nu^+$, where $\nu$ is given by the 
  preceding claim.
  
  Let $\theta$ be a sufficiently large regular cardinal (in particular, we want $T \in H(\theta)$), 
  and let $\mc Z$ be the set of all $M \prec (H(\theta), \in, T)$ such that 
  \begin{itemize}
    \item $\vert M\vert  = \mu$;
    \item there is a strictly $\subseteq$-increasing sequence $\langle M_\eta \mid \eta < \mu \rangle$ 
    such that
    \begin{itemize}
      \item $M = \bigcup_{\eta < \mu} M_\eta$;
      \item for all $\eta < \mu$, $M_\eta \in M$.
    \end{itemize}
  \end{itemize}
  Then $\mc Z$ is stationary in $\power_{\mu^+}H(\theta)$. For each $M \in \power_{\mu^+}H(\theta)$, 
  use the fact that $\Lambda$ is $\mu^+$-directed and $\vert M\vert  = \mu$ to find $v_M \in \Lambda$ such that 
  $u <_\Lambda v_M$ for all $u \in M \cap \Lambda$.
  
  Temporarily fix an $M \in \mc Z$. Since $\vert T_{v_M}\vert  < \mu$ and $\mu$ is regular, we can find an 
  $\eta_M < \mu$ such that, for all distinct $t_0, t_1 \in T_{v_M}$, if there is $u \in M \cap \Lambda$ 
  such that $t_0 \restriction u \neq t_1 \restriction u$, then there is such a $u$ in 
  $M_{\eta_M} \cap \Lambda$. Since $M_{\eta_M} \in M$ and $\vert M_{\eta_M}\vert  \leq \mu$, we can 
  find $u_M \in M \cap \Lambda$ such that $u <_\Lambda u_M$ for all $u \in M_{\eta_M} \cap \Lambda$. 
  
  \begin{claim} \label{nonsplitting_claim}
    For all $v \in M \cap \Lambda$ such that $u_M <_\Lambda v$, $T$ does not split between $u_M$ and 
    $v$.
  \end{claim} 
  
  \begin{proof}
    Suppose for sake of contradiction that $v \in M$, $s_0$ and $s_1$ are distinct elements of 
    $T_v$, and $s_0 \restriction u_M = s_1 \restriction u_M$. Since $T$ is well-pruned, we 
    can find $t_0, t_1 \in T_{v_M}$ such that $t_0 \restriction v = s_0$ and $t_1 \restriction v = 
    s_1$. Then, by our choice of $\eta_M$, there is $u \in M_{\eta_M} \cap \Lambda$ such that 
    $t_0 \restriction u \neq t_1 \restriction u$. But $u < u_M$ and $t_0 \restriction u_M = 
    s_0 \restriction u_M = s_1 \restriction u_M = t_1 \restriction u_M$, so $t_0 \restriction u 
    = (t_0 \restriction u_M) \restriction u = (t_1 \restriction u_M) \restriction u = t_1 \restriction u$, 
    which is a contradiction.
  \end{proof}
  
  The function that takes each $M \in \mc Z$ to $u_M$ is a regressive map, so, since $\mc Z$ is stationary 
  in $\power_{\mu^+}H(\theta)$, we can find a stationary $\mc Z^* \subseteq \mc Z$ and a single 
  $u^* \in \Lambda$ such that $u_M = u^*$ for all $M \in \mc Z^*$. Choose an arbitrary $s^* \in 
  T_{u^*}$, and define a function $b \in \prod_{u \in \Lambda} T_u$ as follows. For each $u \in \Lambda$, 
  find $v \in \Lambda$ such that $u, u^* <_\Lambda v$, use the fact that $T$ is well-pruned to find 
  $t \in T_v$ such that $s^* <_T t$, and let $b(u) = t \restriction u$.
  
  We claim that $b$ is a cofinal branch through $T$, which will complete the proof of the lemma. Suppose 
  for sake of contradiction that there are $u_0 <_\Lambda u_1$ such that $b(u_0) \not<_T b(u_1)$. 
  Recalling the definition of $b$, for each $i < 2$ find $v_i \in \Lambda$ and $t_i \in T_{v_i}$ such 
  that $u_i, u^* <_\Lambda v_i$, $s^* <_T t_i$, and $b(u_i) = t_i \restriction u_i$. Then 
  find $v \in \Lambda$ such that $v_0, v_1 <_\Lambda v$ and $t_0^*, t_1^* \in T_v$ such that 
  $t_0 <_T t_0^*$ and $t_1 <_T t_1^*$. Then $b(u_0) <_T t_0^*$ and $b(u_1) <_T t_1^*$, so, 
  since $b(u_0) \not<_T b(u_1)$, it follows that $t_0^* \neq t_1^*$. Moreover, we know that 
  $t_0^* \restriction u^* = s^* = t_1^* \restriction u^*$. Since $u^* <_T v$, it follows that $T$ 
  splits between $u^*$ and $v$. Now fix $M \in \mc Z^*$ such that $v \in M$. Since $u_M = u^*$, 
  Claim \ref{nonsplitting_claim} implies that $T$ does not split between $u^*$ and $v$, which 
  is the desired contradiction.
\end{proof}

\begin{proposition}
  Suppose that $\kappa$ is strongly compact and $\Lambda$ is $\kappa$-directed. Then $\TP_\kappa(\Lambda)$ 
  holds.
\end{proposition}

\begin{proof}
  Let $T$ be a $\kappa$-$\Lambda$-tree, and let $j:V \rightarrow M$ be an elementary embedding with 
  critical point $\kappa$ such that $j(\kappa) > \vert \Lambda\vert $ and there is $X \subseteq j(\Lambda)$ in 
  $M$ such that $M \models$``$\vert X\vert  < j(\kappa)$" and $j''\Lambda \subseteq X$. In $M$, let
  $j(T) = (\langle \bar{T}_w \mid w \in j(\Lambda) \rangle, <_{j(T)})$. Note that, for all 
  $u \in \Lambda$, since $\vert T_u\vert  < \kappa = \mathrm{crit}(j)$, we have 
  $\bar{T}_{j(u)} = j''T_u$. In $M$, $j(\Lambda)$ is $j(\kappa)$-directed, so we can find 
  $w \in j(\Lambda)$ such that $v <_{j(\Lambda)} w$ for all $v \in X$. Fix an arbitrary 
  $t \in \bar{T}_w$. Now define a function $b \in \prod_{u \in \Lambda} T_u$ by letting $b(u)$
  be the unique $s \in T_u$ such that $j(s) = t \restriction j(u)$ for all $u \in \Lambda$. 
  Then whenever $u_0 <_\Lambda u_1$, we know that, in $M$, we have $j(b(u_0)) = t \restriction j(u_0)$ 
  and $j(b(u_1)) = t \restriction j(u_1)$. Therefore, $j(b(u_0)) <_{j(T)} j(b(u_1))$, so, by 
  elementarity, $b(u_0) <_T b(u_1)$. Therefore, $b$ is a cofinal branch through $T$.
\end{proof}

\section{Background on two-cardinal combinatorics} \label{two_cardinal_section}

As noted in the previous section, many of our results in this paper are in the specific context of 
the partial order $(\power_\kappa \lambda, \subseteq)$. We therefore briefly review some of the relevant 
combinatorial definitions and facts about $\power_\kappa \lambda$.
For this section, let $X$ denote an arbitrary set with $\kappa \leq \vert X\vert $.

\begin{definition}
  Suppose that $\mc C \subseteq \power_\kappa X$.
  \begin{enumerate}
    \item $C$ is \emph{closed} if whenever $D \subseteq C$ is such that $\vert D\vert  < \kappa$ and 
    $D$ is linearly ordered by $\subseteq$, we have $\bigcup D \in C$;
    \item $C$ is \emph{strongly closed} if whenever $D \subseteq C$ and $\vert D\vert  < \kappa$, we have 
    $\bigcup D \in C$;
    \item $C$ is \emph{cofinal} if for all $x \in \power_\kappa \lambda$, there is $y \in C$ such 
    that $x \subseteq y$;
    \item $C$ is a \emph{club} in $\power_\kappa X$ if it is closed and cofinal;
    \item $C$ is a \emph{strong club} in $\power_\kappa X$ if it is strongly closed and cofinal.
  \end{enumerate}
\end{definition}

\begin{definition}
  The \emph{club filter} on $\power_\kappa X$, denoted $\CF_{\kappa, X}$, is the set of all 
  $B \subseteq \power_\kappa X$ for which there is a club $C \subseteq \power_\kappa X$ 
  such that $C \subseteq B$. Similarly, the \emph{strong club filter} on $\power_\kappa X$, 
  denoted $\SCF_{\kappa, X}$, is the set of all $B \subseteq \power_\kappa X$ for which 
  there is a strong club $C \subseteq \power_\kappa X$ such that $C \subseteq B$. The dual ideals 
  to the club filter and the strong club filter are denoted $\NS_{\kappa, X}$ and 
  $\SNS_{\kappa, X}$, respectively. Elements of $\NS_{\kappa, X}^+$ and 
  $\SNS_{\kappa, X}^+$ are called \emph{stationary} and \emph{weakly stationary} subsets of 
  $\power_\kappa X$, respectively.
\end{definition}

Given a set $x \subseteq \power_\kappa X$ and a function $f: X \rightarrow 
\power_\kappa X$, we say that $x$ is \emph{closed under $f$} if $f(a) \subseteq x$ for 
all $a \in x$. Similarly, if $g:[X]^2 \rightarrow \power_\kappa X$, then $x$ is 
closed under $g$ if $g(a) \subseteq x$ for all $a \in [x]^2$.
The following proposition is immediate.

\begin{proposition}
  Suppose that $f:X \rightarrow \power_\kappa X$ is a function. Then the set 
  $\{x \in \power_\kappa X \mid x \text{ is closed under } f\}$ is a strong club in 
  $\power_\kappa X$. In particular, if $\mc Y \subseteq \power_\kappa X$ is 
  weakly stationary, then there is $x \in \mc Y$ such that $x$ is closed under $f$.
\end{proposition}

The following characterization of $\CF_{\kappa, X}$ is due to Menas \cite{menas}.

\begin{proposition} \label{menas_prop}
  If $g:[X]^2 \rightarrow \power_\kappa X$ is a function, then the set 
  \[
  C_g := \{x \in \power_\kappa X \mid x \text{ is infinite and closed under } g\}
  \]
  is a club in $\power_\kappa X$. Moreover, for any club $C$ in $\power_\kappa X$, 
  there is $g:[X]^2 \rightarrow \power_\kappa X$ such that $C_g \subseteq C$.
\end{proposition}

Given sets $X \subseteq X'$ and a subset $S \subseteq \power_\kappa X$, let 
$S \uparrow \power_\kappa X' := \{y \in \power_\kappa X' \mid y \cap X \in S\}$. 
Dually, if $S' \subseteq \power_\kappa X'$, let $S \downarrow \power_\kappa X := 
\{y' \cap X \mid y' \in S'\}$. The following facts are standard.

\begin{proposition} \label{club_transfer_prop}
  Suppose that $X \subseteq X'$ are sets such that $\kappa \leq \vert X\vert $, $S \subseteq \power_\kappa X$, and 
  $S' \subseteq \power_\kappa X'$.
  \begin{enumerate}
    \item If $S$ is a (strong) club in $\power_\kappa X$, then $S \uparrow \power_\kappa X'$ is a 
    (strong) club in $\power_\kappa X'$.
    \item If $S'$ is a strong club in $\power_\kappa X'$, then $S \downarrow \power_\kappa X$ is a 
    strong club in $\power_\kappa X$.
    \item If $S$ is a club in $\power_\kappa X'$, then $S \downarrow \power_\kappa X$ contains a 
    club in $\power_\kappa X$.
    \item If $S$ is (weakly) stationary in $\power_\kappa X$, then $S \uparrow \power_\kappa X'$ is 
    (weakly) stationary in $\power_\kappa X'$.
    \item If $S'$ is (weakly) stationary in $\power_\kappa X'$, then $S' \downarrow \power_\kappa X$ 
    is (weakly) stationary in $\power_\kappa X$.
  \end{enumerate}
\end{proposition}

\section{Strong tree properties and guessing models} \label{strong_tree_section}

In this section, we review the basic definitions and facts about two-cardinal tree properties and 
guessing models. These definitions and facts are largely generalizations of definitions and results 
from \cite{weiss} and \cite{viale_weiss}. The end of the section contains some new results, separating 
certain two-cardinal tree properties and answering a question of Fontanella and Matet.

Let $\kappa$ be a regular uncountable cardinal, and let $X$ be a set with $\vert X\vert  \geq \kappa$. 
We say that a sequence $\seq{d_x}{x\in\power_\kappa X}$ is a $(\kappa,X)$-list if $d_x\sub x$ for all $x\in \power_\kappa X$.

\begin{definition} \label{thin_slender_def}
Assume that $D=\seq{d_x}{x\in\power_\kappa X}$ is a $(\kappa,X)$-list and $\mc Y \subseteq \power_\kappa X$.
\bce[(i)] 
\item We say that $D$ is $\mc Y$-\emph{thin} if there is a closed unbounded set $C\sub \power_\kappa X$ such that $\vert \set{d_x\cap y}{y\sub x\in \power_\kappa X}\vert <\kappa$ for every $y\in C \cap \mc Y$.

\item  \label{slender_item} 
Let $\mu\le\kappa$ be an uncountable cardinal. We say that $D$ is $\mu$-$\mc Y$-\emph{slender} if for all sufficiently large $\theta$ there is a club $C\sub\power_\kappa H(\theta)$ such that for all $M\in C$ and all $y\in M\cap \power_\mu X$, if $M \cap X \in \mc Y$, then $d_{M\cap X}\cap y\in M$.
\ece
Here and in all similar later situations, if $\mc Y = \power_\kappa X$, then we will typically 
omit mention of $\mc Y$ and simply refer to $D$ as being \emph{thin} or \emph{$\mu$-slender}.
\end{definition}

The following fact is easily established (cf.\ \cite[Proposition 2.2]{weiss}).

\begin{fact} \label{thin_slender_fact}
  Assume that $D$ is a $(\kappa,X)$-list and $\mc Y \subseteq \power_\kappa X$. 
  If $D$ is $\mc Y$-thin, then it is $\kappa$-$\mc Y$-slender.
\end{fact}

The following proposition follows almost immediately from the definitions but is often useful, 
showing that the value of $\theta$ in Clause (\ref{slender_item}) of Definition 
\ref{thin_slender_def} can always be taken to be any cardinal $\theta$ for which 
$\power_\mu X \subseteq H(\theta)$.

\begin{proposition}
  Suppose that $\mu \leq \kappa$ is an infinite cardinal,
  $D = \langle d_x \mid x \in \power_\kappa X \rangle$ is a $(\kappa, X)$-list, 
  $\mc Y \subseteq \power_\kappa X$, and $\theta$ is a cardinal such that 
  $\power_\mu X \subseteq H(\theta)$. Then the following are equivalent.
  \begin{enumerate}
    \item $D$ is $\mu$-$\mc Y$-slender.
    \item There is a club $C \subseteq \power_\kappa H(\theta)$ such that for all 
    $M \in C$ and all $y \in M \cap \power_\mu X$, if $M \cap X \in \mc Y$, then 
    $d_{M \cap X} \cap y \in M$.
  \end{enumerate}
\end{proposition}

\begin{proof}
  Suppose that $\theta_0 \leq \theta_1 \leq \theta_2$ are cardinals for which 
  $\power_\mu X \subseteq H(\theta_0)$, and suppose that $C \subseteq \power_\kappa 
  H(\theta_1)$ is a club witnessing the $\mu$-$\mc Y$-slenderness of $D$ (i.e., 
  $C$ is as in Clause (\ref{slender_item}) of Definition \ref{thin_slender_def}). Then it is 
  easily verified that $C \downarrow H(\theta_0)$ and $C \uparrow H(\theta_2)$ also witness 
  the $\mu$-$\mc Y$-slenderness of $D$ (recall Proposition \ref{club_transfer_prop}). The 
  proposition is then immediate.
\end{proof}

\begin{definition}
Assume that $D=\seq{d_x}{x\in\power_\kappa X}$ is a $(\kappa,X)$-list, $\mc Y \subseteq 
\power_\kappa X$ is stationary, and $d\sub X$.
\bce[(i)]
\item We say that $d$ is a \emph{cofinal branch} of $D$ if for all $x\in\power_\kappa X$ there is $z_x\supseteq x$ such that $d\cap x=d_{z_x}\cap x$.
\item We say that $d$ is a \emph{$\mc Y$-ineffable branch} of $D$ if the set $\set{x\in \mc Y}{d\cap x=d_x}$ is stationary. (Again, we will omit mention of $\mc Y$ if $\mc Y = \power_\kappa X$.)
\ece
\end{definition}

\begin{remark} \label{list_to_tree_remark}
  Given a $(\kappa, X)$-list, there is a canonical way of generating a $\power_\kappa X$-tree 
  (in the sense of Section \ref{generalized_tree_section}) from it, where 
  $\power_\kappa X$ is seen as a poset ordered by $\subseteq$. Namely, fix a 
  $(\kappa, X)$-list $D = \langle d_x \mid x \in \power_\kappa X \rangle$, and define a 
  $\power_\kappa X$-tree $T = (\langle T_x \mid x \in \power_\kappa X \rangle, <_T)$ as 
  follows. First, for each $x \in \power_\kappa X$, let $T_x := \{d_y \cap x \mid y \in 
  \power_\kappa X \text{ and } x \subseteq y\}$. Given $x \subsetneq y$ in $\power_\kappa 
  X$, $s \in T_x$, and $t \in T_y$, we set $s <_T t$ if and only if $s = t \cap x$. The 
  following are then immediate:
  \begin{itemize}
    \item $T$ is a $\power_\kappa X$-tree;
    \item for every $\mc Y \subseteq \power_\kappa X$, $D$ is $\mc Y$-thin if and only if 
    $T$ is $\mc Y$-$\kappa$-thin;
    \item $T$ has a cofinal branch if and only if $D$ has a cofinal branch.
  \end{itemize}
\end{remark}

\begin{definition}
  If $D = \langle d_x \mid x \in \power_\kappa X \rangle$ is a $(\kappa, X)$-list, then we let 
  $\width(D)$ denote $\width(T)$, where $T$ is the $\power_\kappa X$-tree generated from $D$ as in 
  Remark \ref{list_to_tree_remark}.
\end{definition}


\begin{definition} \label{tp_def}
Assume that $\mu\le\kappa$ is regular and $\mc Y \subseteq \power_\kappa X$ is stationary. We say that 
\bce[(i)]
\item the $(\kappa,X)$-tree property holds on $\mc Y$, denoted $\TP_{\mc Y}(\kappa,X)$, if every 
$\mc Y$-thin $(\kappa,X)$-list has a cofinal branch.

\item the $(\kappa,X)$-ineffable tree property holds on $\mc Y$, denoted $\ITP_{\mc Y}(\kappa,X)$, if every $\mc Y$-thin $(\kappa,X)$-list has a $\mc Y$-ineffable branch.
\item the $(\mu,\kappa,X)$-slender tree property holds on $\mc Y$, denoted $\SP_{\mc Y}(\mu,\kappa,X)$, if every $\mu$-$\mc Y$-slender $(\kappa,X)$-list has a cofinal branch.
\item the $(\mu,\kappa,X)$-ineffable slender tree property holds on $\mc Y$, denoted $\ISP_{\mc Y}(\mu,\kappa,X)$, if every $\mu$-$\mc Y$-slender $(\kappa,X)$-list has a $\mc Y$-ineffable branch.
\ece
\end{definition}

\begin{remark} \label{tp_convention_remark}
  In order to ease notation, we introduce a couple of conventions. As before, if 
  mention of $\mc Y$ is omitted in any of the principles from Definition \ref{tp_def}, 
  then it should be understood that $\mc Y = \power_\kappa X$. We will use notations 
  such as $\mathsf{(I)TP}(\kappa, \geq \kappa)$ (resp.\ $\mathsf{(I)SP}(\mu, 
  \kappa, \geq \kappa)$) to assert that $\mathsf{(I)TP}(\kappa, \lambda)$ (resp.\ 
  $\mathsf{(I)SP}(\mu, \kappa, \lambda)$) holds for all $\lambda \geq \kappa$. 
  Finally, in the principles $\mathsf{(I)SP}_{\mc Y}(\mu, \kappa, \lambda)$, the value 
  of $\mu$ that has most often been considered in the literature is $\omega_1$; we will 
  therefore use $\mathsf{(I)SP}_\kappa$ to denote $\mathsf{(I)SP}(\omega_1, \kappa, \geq 
  \kappa)$.
\end{remark}

If $f:X_0 \rightarrow X_1$ is a bijection, $\mc Y_0 \subseteq \power_\kappa X_0$ is 
stationary, and $\mc Y_1 := \{f[x] \mid x \in \mc Y_0\}$, then $\TP_{\mc Y_0}(\kappa, X_0)$ is equivalent 
to $\TP_{\mc Y_1}(\kappa, X_1)$, $\ITP_{\mc Y_0}(\mu, \kappa, X_0)$ is equivalent to 
$\ITP_{\mc Y_1}(\mu, \kappa, X_1)$, and similarly for $\SP$ and $\ISP$. We will therefore 
not lose any generality by stating our results in the context in which $X$ is an infinite 
cardinal, which is what is typically seen in the literature. There are instances, though, 
in which it is more convenient to work with, e.g., $(\kappa, H(\theta))$-lists, so 
we have opted for the more general terminology and notation. Also, if $C \subseteq 
\power_\kappa X_0$ is a club, then $\TP_{\mc Y_0}(\kappa, X_0)$ is equivalent to 
$\TP_{\mc Y_0 \cap C}(\kappa, X_0)$, and similarly for the other principles.

Note that $\mathsf{(I)SP}_{\mc Y}(\mu,\kappa,\lambda)$ implies $\mathsf{(I)SP}_{\mc Y}(\nu,\kappa,\lambda)$ for $\kappa \ge \nu\ge\mu$. The converse does not hold: In Theorem \ref{th:SP} we show that $\SP(\omega_1,\omega_2,\omega_2)$ implies that there are no weak $\omega_1$-Kurepa trees, while in Theorem \ref{th:SPconverse} we show that $\ISP(\omega_2,\omega_2,\lambda)$, and hence also the weaker $\SP(\omega_2,\omega_2,\lambda)$, is consistent with the existence of a (thin) $\omega_1$-Kurepa tree.\footnote{For concreteness, we formulate the result of $\omega_2$, but it can be easily generalized to an arbitrary double successor of a regular cardinal.} Also, by Fact \ref{thin_slender_fact}, $\mathsf{(I)SP}_{\mc Y}(\kappa, 
\kappa, \lambda)$ implies $\mathsf{(I)TP}(\kappa, \lambda)$. We will see in Subsection 
\ref{approachability_subsection} below that these implications are also in general not reversible. 

Note additionally that there is monotonicity in the last coordinate of these principles: if $\lambda' \geq \lambda \geq \kappa$ 
and $\mc Y' = \mc Y \uparrow \power_\kappa \lambda'$, 
then $\mathsf{(I)SP}_{\mc Y'}(\mu, \kappa, \lambda')$ implies $\mathsf{(I)SP}_{\mc Y}(\mu, \kappa, 
\lambda)$, and $\mathsf{(I)TP}_{\mc Y'}(\kappa, \lambda')$ implies $\mathsf{(I)TP}_{\mc Y}(\kappa, \lambda)$.

We now recall a useful reformulation of instances of $\ISP(\ldots)$ in terms of \emph{guessing models}. 
We first introduce some basic definitions. Terminology regarding guessing models is slightly inconsistent 
across sources; we will primarily be following the terminology and notation from 
\cite{viale_guessing_models}. We note, however, that our definition of a \emph{$(\mu, M)$-approximated 
set} is formally weaker than the standard definition. It is easily seen to be equivalent if $M$ is closed 
under pairwise intersections of its elements (in particular, if $M \prec H(\theta)$), but since we 
will sometimes want to apply the definition to sets $M$ that are not elementary submodels of $H(\theta)$, 
our weaker definition seems more appropriate.

\begin{definition} \label{guessing_model_def} 
  Suppose that $\theta$ is a sufficiently large regular cardinal and $M \subseteq H(\theta)$.
  \begin{enumerate}
    \item Given a set $x \in M$, a subset $d \subseteq x$, and an uncountable cardinal $\mu$, we say that 
    \begin{enumerate}
      \item $d$ is \emph{$(\mu, M)$-approximated} if, for every $z \in M \cap \power_\mu(x)$, 
      there is $e \in M$ such that $d \cap z = e \cap z$;
      \item $d$ is \emph{$M$-guessed} if there is $e \in M$ such that $d \cap M = e \cap M$.
    \end{enumerate}
    \item $M$ is a \emph{$\mu$-guessing model for $x$} if $M \prec H(\theta)$ and
    every $(\mu, M)$-approximated subset of $x$ is $M$-guessed.
    \item $M$ is a \emph{$\mu$-guessing model} if, for every $x \in M$, it is a $\mu$-guessing 
    model for $x$.
    \item Given uncountable cardinals $\mu \leq \kappa \leq \theta$ with $\kappa$ and 
    $\theta$ regular, and given $x \in H(\theta)$, let $\mc G^x_{\mu, \kappa} H(\theta)$ 
    denote the set of $M \in \power_\kappa H(\theta)$ such that $M$ is a $\mu$-guessing model 
    for $x$. Let $\mc G_{\mu,\kappa} H(\theta)$ denote the set of $M \in \power_\kappa 
    H(\theta)$ such that $M$ is a $\mu$-guessing model.
    \item Suppose that $\mu \leq \kappa \leq \theta$ are uncountable cardinals with $\kappa$ 
    and $\theta$ regular, and that 
    $\mc Y \subseteq \power_\kappa H(\theta)$ is stationary. Then $\GMP_{\mc Y}(\mu, \kappa, \theta)$ 
    is the assertion that $\mc G_{\mu,\kappa} H(\theta) \cap \mc Y$ is 
    stationary in $\power_\kappa H(\theta)$, i.e., $\mc Y$ contains stationarily many 
    $\mu$-guessing models.
  \end{enumerate}
\end{definition}

\begin{remark} \label{gmp_convention_remark}
  As was the case with Definition \ref{tp_def}, we introduce some notational conveniences. 
  If $\mc Y$ is omitted from $\GMP_{\mc Y}(\mu, \kappa, \theta)$, then it should be 
  understood to be $\power_\kappa H(\theta)$. We let $\GMP(\mu, \kappa, \geq \kappa)$ 
  denote the assertion that $\GMP(\mu, \kappa, \theta)$ holds for every regular 
  $\theta \geq \kappa$. Again, the case $\mu = \omega_1$ is the most prominent in the 
  literature; we will simply say that a model $M$ is a \emph{guessing model} to mean 
  that it is an $\omega_1$-guessing model, and we will write 
  $\mc G^x_\kappa H(\theta)$ and $\mc G_\kappa H(\theta)$ instead of 
  $\mc G^x_{\omega_1,\kappa} H(\theta)$ or $\mc G_{\omega_1, \kappa} H(\theta)$.
\end{remark}

The following proposition is immediate from the definitions.

\begin{proposition} \label{transfer_prop}
  Suppose that $\mu \leq \kappa \leq \theta \leq \theta'$ are uncountable cardinals, 
  with $\kappa$, $\theta$, and $\theta'$ regular. 
  \begin{enumerate}
    \item \label{transfer_clause_1} Suppose that $x \in M \prec H(\theta)$, $M' \prec H(\theta')$, 
    and $M' \cap H(\theta) = M$. Then $M$ is a $\mu$-guessing model for $x$ if and only if 
    $M'$ is a $\mu$-guessing model for $x$. In particular, $\mc G^x_{\mu, \kappa} H(\theta)$ 
    is stationary in $\power_\kappa H(\theta)$ if and only if $\mc G^x_{\mu, \kappa} 
    H(\theta')$ is stationary in $\power_\kappa H(\theta')$.
    \item Suppose that $M \prec H(\theta)$, $M' \prec H(\theta')$, and $M' \cap H(\theta) = 
    M$. If $M'$ is a $\mu$-guessing model, then $M$ is also a $\mu$-guessing model. 
    In particular, $\GMP(\mu, \kappa, \theta')$ implies $\GMP(\mu, \kappa, \theta)$.
  \end{enumerate}
\end{proposition}

The proofs of the following propositions are essentially the same as those of 
\cite[Propositions 3.2 and 3.3]{viale_weiss}; we include them for completeness.

\begin{proposition} \label{guessing_isp_prop_1}
  Let $\mu \leq \kappa \leq \theta$ be regular uncountable cardinals, and let 
  $\mc Y \subseteq \power_\kappa H(\theta)$ be stationary.
   If $\ISP_{\mc Y}(\mu, \kappa, 
  H(\theta))$ holds, then $\GMP_{\mc Y}(\mu, \kappa, \theta)$ holds.
\end{proposition}

\begin{proof}
  Suppose for sake of contradiction that there is a club $C \subseteq \power_\kappa H(\theta)$ such 
  that every element of $C \cap \mc Y$ is not a $\mu$-guessing model.
  Then for every $M \in C \cap \mc Y$, we can fix a 
  set $z_M \in M$ and $d_M \subseteq z_M$ such that $d_M$ is $(\mu, M)$-approximated but not 
  $M$-guessed. The same is easily seen to also be true for $d_M \cap M$, so we can assume that 
  $d_M \subseteq M$. For $M \in \power_\kappa H(\theta) \setminus C \cap \mc Y$, 
  let $d_M$ be an arbitrary subset of
  $M$. This defines a $(\kappa, H(\theta))$-list $D := \langle d_M \mid M \in 
  \power_\kappa H(\theta) \rangle$.
  
  We claim that $D$ is $\mu$-$\mc Y$-slender. Let $\theta' > \vert H(\theta)\vert $ be a regular cardinal, and
  let $C' := \{M' \in C \uparrow \power_\kappa H(\theta') \mid M' \prec H(\theta') 
  \}$. Then $C'$ is a club in $\power_\kappa H(\theta')$, and the fact that $C'$ 
  witnesses the $\mu$-$\mc Y$-slenderness of $D$ follows immediately from the fact 
  that $d_M$ is $(\mu, M)$-approximated for all $M \in C \cap \mc Y$.
  
  We can therefore apply $\ISP_{\mc Y}(\mu, \kappa, H(\theta))$ to find a $\mc Y$-ineffable branch 
  $d$ of $D$. Let $S_0:= \{M \in C \cap \mc Y \mid d \cap M = d_M\}$. Then $S_0$ is stationary in 
  $\power_\kappa H(\theta)$, so, by an application of Fodor's lemma, we can fix a stationary $S_1 \subseteq S_0$ and a fixed set $z$ 
  such that $z_M = z$ for all $M \in S_1$. Since $S_1$ is $\subseteq$-cofinal in
  $\power_\kappa H(\theta)$, it follows that $d \subseteq z$, and hence $d \in H(\theta)$.
  We can therefore find $M \in S_1$ such that $d \in M$. But then $d_M \cap M = d \cap M$, 
  contradicting the fact that $d_M$ is not $M$-guessed.
\end{proof}

\begin{proposition} \label{guessing_isp_prop_2}
  Let $\mu \leq \kappa \leq \lambda$ be regular uncountable cardinals, and let 
  $\mc Y \subseteq \power_\kappa \lambda$ be stationary. Suppose that there is a 
  regular cardinal $\theta > \lambda^{<\kappa}$ such that 
  $\GMP_{\mc Y'}(\mu, \kappa, \theta)$ holds, 
  where $\mc Y' := \mc Y \uparrow \power_\kappa H(\theta)$.
  Then $\ISP_{\mc Y}(\mu, \kappa, \lambda)$ holds.
\end{proposition}

\begin{proof}
	Let $D = \langle d_x \mid x \in \power_\kappa \lambda \rangle$ be a $\mu$-$\mc Y$-slender 
  $(\kappa, \lambda)$-list. We will find a $\mc Y$-ineffable branch for $D$. Let $C \subseteq \power_\kappa 
  H(\lambda^+)$ be a club witnessing that $D$ is $\mu$-$\mc Y$-slender. By our hypothesis, we can 
  find an $M \in \power_\kappa H(\theta)$ such that $D, \mc Y \in M$, $M \cap H(\lambda^+) \in C$,
  $M \cap \lambda \in \mc Y$, and $M$ is a $\mu$-guessing model for $\lambda$.
  By the fact that $M \cap H(\lambda^+) \in C$, we have $d_{M \cap \lambda} \cap y \in M$ for all $y \in M \cap \power_\mu \lambda$, which directly implies that $d_{M \cap \lambda}$ is $(\mu, M)$-approximated. 
  Since $M$ is a $\mu$-guessing model, there is $e \in M$ such that $e \cap M = d_{M \cap \lambda} 
  \cap M = d_{M \cap \lambda}$, where the second equality follows from the fact 
  that $D$ is a $(\kappa,\lambda)$-list and hence $d_{M \cap \lambda} 
  \subseteq M \cap \lambda$. Note that, since $d_{M \cap \lambda} \subseteq \lambda$ and 
  $\lambda \in M$, it follows that $e \subseteq \lambda$.
  
  We claim that $e$ is a $\mc Y$-ineffable branch for $D$. Let $S := \{x \in \mc Y
  \mid e \cap x = d_x\}$, and note that $M \cap \lambda \in S$. If $S$ were not stationary, then 
  there would be a club $E \subseteq \power_\kappa \lambda$ such that $E \cap S = \emptyset$. 
  Since everything needed to define $S$ is in $M$, we can assume by elementarity that $E \in M$. 
  But then, since $E$ is a club, we have $M \cap \lambda \in E$, and we already saw that $M \cap \lambda 
  \in S$, contradicting the assumption that $E \cap S = \emptyset$. Therefore, $e$ is indeed a 
  $\mc Y$-ineffable branch for $D$. 
\end{proof}

\begin{corollary}\label{isp_guessing_cor}
  Suppose that $\mu \leq \kappa$ are regular uncountable cardinals. Then the following are equivalent:
  \begin{enumerate}
    \item $\ISP(\mu, \kappa, \geq \kappa)$;
    \item $\GMP(\mu, \kappa, \geq \kappa)$.
  \end{enumerate}
\end{corollary}

\begin{proof}
  This is immediate from Propositions \ref{guessing_isp_prop_1} and \ref{guessing_isp_prop_2}.
\end{proof}

Note that there is a local asymmetry between Propositions \ref{guessing_isp_prop_1} and 
\ref{guessing_isp_prop_2}: by Proposition \ref{guessing_isp_prop_1}, 
$\ISP(\mu, \kappa, \vert H(\theta)\vert )$ implies $\GMP(\mu, \kappa, H(\theta))$, but 
Proposition \ref{guessing_isp_prop_2} does not provide an exact converse to this fact. 
Instead, we must assume that $\GMP(\mu, \kappa, H(\theta'))$ holds for some $\theta' > \vert H(\theta)\vert $ 
to conclude that $\ISP(\mu, \kappa, \vert H(\theta)\vert )$ holds. We now show that this is 
necessary; in fact, $\GMP(\mu, \kappa, H(\theta))$ does not in general imply 
even $\ITP(\kappa, \vert H(\theta)\vert )$. We first need the following standard proposition.

\begin{proposition} \label{stat_pres_prop}
  Suppose that $\kappa \leq \theta$ are regular uncountable cardinals, 
  $S \subseteq \power_\kappa H(\theta)$ is stationary, and $\P$ is a 
  $\vert H(\theta)\vert $-strategically closed forcing notion. Then $S$ remains a stationary 
  subset of $\power_\kappa H(\theta)$ in $V^{\P}$.
\end{proposition}

\begin{proof}
  Note first that, since $\P$ is $\vert H(\theta)\vert $-strategically closed and thus 
  certainly ${<}\theta$-distributive, we have $H(\theta)^V = H(\theta)^{V[G]}$ 
  and $(\power_\kappa H(\theta))^V = (\power_\kappa H(\theta))^{V[G]}$. 
  By Proposition \ref{menas_prop}, it suffices to show that, for every $p \in \P$ and 
  every $\P$-name $\dot{g}$ for a function from $[H(\theta)]^2$ to 
  $\power_\kappa H(\theta)$, there is $q \leq_{\P} p$ and $M \in S$ such that 
  $q \Vdash_{\P}$``$M \text{ is closed under } \dot{g}$''. To this end, fix such a $p$ and 
  $\dot{g}$, and enumerate $H(\theta)$ as $\langle x_\alpha \mid \alpha < 
  \vert H(\theta)\vert  \rangle$. For all $\beta < \vert H(\theta)\vert $, let 
  $X_\beta := \{x_\alpha \mid \alpha < \beta\}$. Using the fact that $\P$ is 
  $\vert H(\theta)\vert $-strategically closed, 
  recursively construct a $\leq_{\P}$-decreasing sequence $\langle p_\beta \mid 
  \beta < \vert H(\theta)\vert  \rangle$ such that $p_0 = p$ and, for all $\beta < 
  \vert H(\theta)\vert $, $p_{\beta + 1}$ decides the value of $\dot{g} \restriction 
  [X_\beta]^2$, say as $g_\beta : [X_\beta]^2 \rightarrow \power_\kappa H(\theta)$.
  Let $g^* = \bigcup_{\beta < \vert H(\theta)\vert } g_\beta$. Then $g^*:[H(\theta)]^2 \rightarrow 
  \power_\kappa H(\theta)$, so we can find $M \in S$ such that $M$ is closed under 
  $g^*$. Note that $\vert H(\theta) \vert = 2^{<\theta}$ and, since $\theta$ 
  is regular, $\cf(2^{<\theta}) \geq \theta \geq \kappa$. Therefore, we 
  can find $\beta < \vert H(\theta)\vert $ large enough so that $M \subseteq X_\beta$. 
  Then $p_\beta \Vdash_{\P}$``$\dot{g} \restriction [M]^2 = g_\beta \restriction [M]^2$'', 
  so $p_\beta$ forces that $M$ is closed under $\dot{g}$.
\end{proof}

\begin{theorem}
  Suppose that $\mu \leq \kappa \leq \theta$ are regular uncountable cardinals, 
  $\vert H(\theta)\vert $ is a regular cardinal, and 
  $\GMP(\mu, \kappa, \theta)$ holds. Then there is a forcing extension, preserving 
  all cofinalities $\leq \vert H(\theta)\vert $, in which $\GMP(\mu, \kappa, \theta) + 
  \square(\vert H(\theta)\vert )$ 
  holds; in particular, $\ITP(\kappa, \vert H(\theta)\vert )$ (and hence $\ISP(\mu, \kappa, \vert H(\theta)\vert )$) fails 
  in this model.
\end{theorem}

\begin{proof}
  Let $\lambda := \vert H(\theta)\vert $.
  Let $\P$ be the standard forcing poset to add a $\square(\lambda)$-sequence by closed 
  initial segments (cf.\ \cite[\textsection 3]{clh_covering} for a precise definition of $\P$ and 
  proofs of relevant facts thereon). $\P$ is $\lambda$-strategically closed, so 
  $(\power_\kappa H(\theta))^V = (\power_\kappa H(\theta))^{V[G]}$. It follows that 
  every $\mu$-guessing model $M \in \power_\kappa H(\theta)$ in $V$ remains a $\mu$-guessing 
  model in $V[G]$, since no new subsets of $M$ are added by forcing with $\P$. In addition, 
  Proposition \ref{stat_pres_prop} implies that the set of such $\mu$-guessing models remains 
  stationary in $V^{\P}$. Therefore, $V^{\P}$ satisfies 
  $\GMP(\mu, \kappa, \theta) + \square(\lambda)$.
  
  To see that $\ITP(\kappa, \lambda)$ fails in $V^{\P}$, simply appeal to 
  \cite[Theorem 4.2]{weiss}, which states that, if $\cf(\lambda) \geq \kappa$ 
  and $\square(\lambda)$ (or even a substantial weakening of $\square(\lambda)$) holds, 
  then $\ITP(\kappa, \lambda)$ fails.
\end{proof}

\subsection{A question of Fontanella and Matet} \label{fontanella_matet_subsection}

Let us take a brief detour to show that the results of Section \ref{generalized_tree_section} 
provide an answer to a question of Fontanella and Matet from 
\cite{fontanella_matet}. In that paper, the authors consider the tree property 
$\TP(\kappa, \lambda)$ as well as an apparent weakening, which they denote 
$\TP^-(\kappa, \lambda)$.

\begin{definition}
  $\TP^-(\kappa, \lambda)$ is the following assertion: If $D = \langle d_x \mid 
  x \in \power_\kappa \lambda \rangle$ is a $(\kappa, \lambda)$-list for which 
  there is a strongly closed, cofinal $C \subseteq \power_\kappa \lambda$ such 
  that $\vert \set{d_x\cap y}{y\sub x\in \power_\kappa\lambda}\vert <\kappa$ for every $y\in C$, 
  then $D$ has a cofinal branch.
\end{definition}

The authors isolate a certain partition relation, denoted $\mathrm{PS}(\kappa, \lambda)$, 
and prove that $\mathrm{PS}(\kappa, \lambda)$ implies $\TP^-(\kappa, \lambda)$. 
They then ask whether $\mathrm{PS}(\kappa, \lambda)$ implies $\TP(\kappa, \lambda)$ 
and, in particular, whether $\TP^-(\kappa, \lambda)$ and $\TP(\kappa, \lambda)$ 
are equivalent. The following corollary gives a positive answer to both questions.

\begin{corollary}
  Suppose that $\TP(\kappa, \lambda)$ holds, and suppose that $D = \langle d_x 
  \mid x \in \power_\kappa \lambda \rangle$ is a $(\kappa, \lambda)$-list for 
  which there is a cofinal $\mc Y \subseteq \power_\kappa \lambda$ such that 
  $\vert \set{d_x\cap y}{y\sub x\in \power_\kappa\lambda}\vert <\kappa$ for every $y\in \mc Y$. 
  Then $D$ has a cofinal branch.
\end{corollary}

\begin{proof}
  Define a $\power_\kappa \lambda$-tree $T = \langle T_x \mid x \in \power_\kappa \lambda 
  \rangle$ from $D$ as in Remark \ref{list_to_tree_remark}. In particular,
  $T_x := \set{d_x\cap y}{y\sub x\in \power_\kappa\lambda}$ for 
  all $x \in \power_\kappa \lambda$. Then $T$ is $\kappa$-$\mc Y$-thin, so, by Proposition 
  \ref{generalized_well_pruned_prop}, we can find a well-pruned subtree $T' = \langle T'_x \mid 
  x \in \power_\kappa \lambda \rangle$ of $T$. The fact that $T'$ is well-pruned implies 
  that $\vert T'_x\vert  \leq \vert T'_y\vert $ for all $x \subseteq y \in \power_\kappa \lambda$. 
  In particular, since $T$ is $\kappa$-$\mc Y$-thin and $\mc Y$ is cofinal in 
  $\power_\kappa \lambda$, we have $\vert T'_x\vert  < \kappa$ 
  for all $x \in \power_\kappa \lambda$. Now let $D' = \langle d'_x 
  \mid x \in \power_\kappa \lambda \rangle$ be a $(\kappa, \lambda)$-list such that 
  $d'_x \in T'_x$ for all $x \in \power_\kappa \lambda$. $D'$ is thin, so, by 
  $\TP(\kappa, \lambda)$, it has a cofinal branch, $d$. By construction, for all 
  $x \in \power_\kappa \lambda$, there is $y \in \power_\kappa \lambda$ such that 
  $y \supseteq x$ and $d'_x = d_y \cap x$; it follows that $d$ is a cofinal branch 
  through $D$, as well.
\end{proof}

\begin{corollary}
  If $\mc Y \subseteq \power_\kappa \lambda$ is cofinal, then $\TP(\kappa, \lambda)$ is 
  equivalent to $\TP_{\mc Y}(\kappa, \lambda)$.
\end{corollary}

\subsection{Approachability and separating $\ISP$ from $\ITP$} \label{approachability_subsection}

It follows immediately from the definitions that, for $\lambda \geq \omega_2$, we have 
\[
\ISP(\omega_1, \omega_2, \lambda) \Rightarrow \ISP(\omega_2, \omega_2, \lambda) 
\Rightarrow \ITP(\omega_2, \lambda),
\] 
and, as mentioned above, we will see in Theorem \ref{th:SPconverse} that $\ISP(\omega_2, \omega_2, \lambda)$ 
does not imply $\ISP(\omega_1, \omega_2, \lambda)$. This raises the natural question of 
whether $\ITP(\omega_2, \lambda)$ implies $\ISP(\omega_2, \omega_2, \lambda)$. In this subsection, 
we answer this question negatively.\footnote{For concreteness, we focus here on $\omega_2$, but 
analogous arguments work at other double successors of regular cardinals.}

We first note that $\ISP(\omega_2, \omega_2, \omega_2)$ implies a failure of approachability.

\begin{proposition}
  Suppose that $\ISP(\omega_2, \omega_2, \omega_2)$ holds. Then $\neg \mathsf{AP}_{\omega_1}$ holds.
\end{proposition}

\begin{proof}
  In \cite[Corollary 4.9]{viale_weiss}, Viale and Weiss show that $\ISP(\omega_1, \omega_2, \omega_2)$ 
  implies $\neg \mathsf{AP}_{\omega_1}$ (for a more detailed proof, using guessing models, see 
  \cite[Proposition 2.6]{cox_krueger_indestructible}). An examination of the proofs in 
  \cite{viale_weiss} and \cite{cox_krueger_indestructible}, though, 
  shows that they only really need $\ISP(\omega_2, \omega_2, \omega_2)$.
\end{proof}

We next show that $\ITP(\omega_2, \lambda)$ is consistent with $\mathsf{AP}_{\omega_1}$.

\begin{proposition}
  Suppose that $\kappa$ is a supercompact cardinal. Then there is a forcing extension in which 
  $\kappa = \omega_2$, $\mathsf{AP}_{\omega_1}$ holds, and $\ITP(\omega_2, \lambda)$ holds for all $\lambda \geq \omega_2$.
\end{proposition}

\begin{proof}
  We use the Mitchell forcing variation $\M_0$ from \cite[\textsection 3.2]{8fold_way}, with the parameter 
  ``$\kappa$" from that paper set to be $\omega$. As shown in \cite[\textsection 3.5]{8fold_way}, 
  in $V^{\M_0}$ we have $\omega_2 \in I[\omega_2]$, i.e., $\mathsf{AP}_{\omega_1}$ holds.
  
  In \cite{fontanella}, Fontanella proves that, in the extension by a slightly different 
  variant of Mitchell forcing, $\ITP(\omega_2, \lambda)$ holds for all $\lambda \geq \omega_2$. 
  The same argument works for $\M_0$; the key point is that $\M_0$, as well as all of its quotients 
  over initial segments of inaccessible length, have the property that there are projections onto 
  them from products $\P \times \Q$, where $\P$ is the forcing to add $\kappa$-many Cohen reals 
  and $\Q$ is countably closed (see Section \ref{mitchell_section} below for more about Mitchell 
  forcing and this property in particular).
\end{proof}

The following is now immediate.

\begin{corollary}
  Suppose that $\kappa$ is a supercompact cardinal. Then there is a forcing extension in which 
  $\ISP(\omega_2, \omega_2, \omega_2)$ fails but $\ITP(\omega_2, \lambda)$ holds for all 
  $\lambda \geq \omega_2$.
\end{corollary}

\section{Slender trees and the almost guessing property} \label{almost_guessing_section}

In this section, we isolate a guessing principle in the style of $\GMP(\ldots)$ that provides 
an alternative formulation of $\SP(\ldots)$ in the same way that $\GMP(\ldots)$ provides 
an alternative formulation of $\ISP(\ldots)$. We begin with a rough heuristic that the reader may or may 
not find helpful. Note that the guessing models witnessing principles of the form $\GMP(\ldots)$ are in 
effect performing two roles. First, they are approximating a given set of interest. Second, they are guessing 
this set, and providing the setting for subsequent elementarity arguments. The fact that $\ISP(\ldots)$ 
produces \emph{ineffable} branches for slender lists, i.e., branches whose small pieces are precisely 
approximated by the elements of the list at stationarily many entries, is what allows us to find guessing 
models that fulfill these two roles simultaneously. When we only have $\SP(\ldots)$, though, and are not 
necessarily able to find ineffable branches, it is possible that these two roles must be pulled apart 
and fulfilled by two different sets, the second a possibly proper subset of the first. We now make 
this heuristic more precise.

\begin{definition}
  Let $\mu \leq \kappa \leq \theta$ be uncountable cardinals, with $\kappa$ and $\theta$ regular, 
  and suppose that $x \in H(\theta)$,
  $S \subseteq \power_\kappa H(\theta)$ is cofinal, and $M \subseteq H(\theta)$. We say that 
  $(M, x)$ is \emph{almost $\mu$-guessed by $S$} if $x \in M$ and, for every $(\mu, M)$-approximated 
  subset $d \subseteq x$, there is an $N \in S$ such that
  \begin{itemize}
    \item $x \in N \subseteq M$; and
    \item $d$ is $N$-guessed.
  \end{itemize}   
\end{definition}

\begin{definition}
  Let $\mu \leq \kappa \leq \theta$ be uncountable cardinals, with $\kappa$ and $\theta$ regular, and let 
  $\mc Y \subseteq \power_\kappa H(\theta)$ be stationary. We say that $\AGP_{\mc Y}(\mu, \kappa, 
  \theta)$ holds if for every cofinal $S \subseteq \power_\kappa H(\theta)$ and 
  every $x \in H(\theta)$, the set of $M \in \mc Y$ such that $(M,x)$ is almost 
  $\mu$-guessed by $S$ is stationary in $\power_\kappa H(\theta)$.
\end{definition}

All of the notational conventions regarding $\GMP_{\mc Y}(\ldots)$ from Remark~\ref{gmp_convention_remark} apply, 
\emph{mutatis mutandis}, to the setting of $\AGP_{\mc Y}(\ldots)$.

\begin{theorem} \label{slender_guessing_thm_1}
  Let $\mu \leq \kappa \leq \theta$ be uncountable cardinals, with $\kappa$ and $\theta$ regular, and let 
  $\mc Y \subseteq \power_\kappa H(\theta)$ be stationary. If $\SP_{\mc Y}(\mu, \kappa, H(\theta))$ 
  holds, then $\AGP_{\mc Y}(\mu, \kappa, \theta)$ holds.
\end{theorem}

\begin{proof}
  Suppose that $\SP_{\mc Y}(\mu, \kappa, H(\theta))$ holds.
  Fix a cofinal $S \subseteq \power_\kappa H(\theta)$ and a set $x \in H(\theta)$, and 
  assume for sake of contradiction that there is a club $C \subseteq \power_\kappa H(\theta)$ 
  such that, for all $M \in C \cap \mc Y$, $(M,x)$ is not almost $\mu$-guessed by $S$. By 
  thinning out $C$ if necessary, we can 
  assume that $x \in M$ for all $M \in C$. Therefore, for each $M \in C \cap \mc Y$, 
  we can fix a set $d_M \subseteq x$ such that $d_M$ is $(\mu, M)$-approximated but there is no 
  set $N \in S$ such that $x \in N$, $N \subseteq M$, and $d_M$ is $N$-guessed. As in the proof of 
  Proposition \ref{guessing_isp_prop_1}, by replacing $d_M$ by $d_M \cap M$ if necessary, we can assume that 
  $d_M \subseteq M$. For $M \in \power_\kappa H(\theta) \setminus C \cap \mc Y$, choose an arbitrary 
  $M^\ast \in C \cap \mc Y$ such that $M \subseteq M^\ast$, and let $d_M := d_{M^\ast} \cap M$.
  
  Let $D := \langle d_M \mid M \in \power_\kappa H(\theta) \rangle$. Again as in the proof of 
  Proposition \ref{guessing_isp_prop_1}, it follows that $D$ is a $\mu$-$\mc Y$-slender 
  $(\kappa, H(\theta))$-list. Therefore, by $\SP_{\mc Y}(\mu, \kappa, H(\theta))$, 
  we can find a cofinal branch $d \subseteq H(\theta)$ through $D$.
  Since $d_M \subseteq x$ for all $M \in \power_\kappa H(\theta)$, it follows that 
  $d \subseteq x$, and hence $d \in H(\theta)$. We can therefore fix an $N \in S$ such 
  that $d,x \in N$ and then fix an $M \in \power_\kappa H(\theta)$ such that $M \supseteq N$ 
  and $d \cap N = d_M \cap N$. We can assume that $M \in C \cap \mc Y$ (if not, we can 
  replace it with the $M^\ast$ used in the definition of $d_M$, since $d_M = d_{M^\ast} 
  \cap M$). But then $d_M$ is $N$-guessed, as witnessed by $d$, contradicting the fact that 
  there is no $N \in S$ such that $x \in N$, $N \subseteq M$, and $d_M$ is $N$-guessed.
\end{proof}

\begin{theorem} \label{slender_guessing_thm_2}
  Let $\mu \leq \kappa \leq \lambda$ be uncountable cardinals, with $\kappa$ and $\theta$ regular, and let 
  $\mc Y \subseteq \power_\kappa \lambda$ be stationary. Suppose that there is a 
  regular cardinal $\theta > \lambda^{<\kappa}$ such that $\AGP_{\mc Y'}(\mu, \kappa, 
  \theta)$ holds, where $\mc Y' := \mc Y \uparrow \power_\kappa H(\theta)$. 
  Then $\SP_{\mc Y}(\mu, \kappa, \lambda)$ holds.
\end{theorem}

\begin{proof}
  Let $D = \langle d_x \mid x \in \power_\kappa \lambda \rangle$ be a $\mu$-$\mc Y$-slender 
  $(\kappa, \lambda)$-list. We will find a cofinal branch for $D$. Let $C \subseteq 
  \power_\kappa H(\lambda^+)$ be a club witnessing that $D$ is $\mu$-$\mc Y$-slender. Let $S$ 
  be the set of 
  $N \prec H(\theta)$ such that $\vert N\vert  < \kappa$, $D \in N$, and $N \cap \kappa \in \kappa$.
  By hypothesis, we can find $M \in \power_\kappa H(\theta) \cap \mc Y'$ such that 
  $M \cap H(\lambda^+) \in C$
  and $(M, \lambda)$ is almost $\mu$-guessed by $S$. Since $M \cap H(\lambda^+) \in C$ 
  and $M \cap H(\lambda) \in \mc Y$, the fact that $C$ witnesses that $D$ is 
  $\mu$-$\mc Y$-slender implies that
  $d_{M \cap \lambda}$ is a $(\mu, M)$-approximated subset of $\lambda$. 
  Therefore, we can find $N \in S$ such that $N \subseteq M$ and $d_{M \cap \lambda}$ is 
  $N$-guessed, i.e., there is $e \in N$ such that $e \cap N = d_{M \cap \lambda} \cap N$.
  
  Since $d_{M \cap \lambda} \subseteq \lambda$, it follows from the elementarity of $N$ 
  that $e \subseteq \lambda$. We claim that $e$ is a cofinal branch of $D$, which will finish 
  the proof. To verify this claim, first fix an arbitrary $x \in \power_\kappa \lambda \cap N$. 
  Since $N \cap \kappa \in \kappa$, it follows that $\vert x\vert  \subseteq N$ and hence, by elementarity, 
  $x \subseteq N$. In particular, we have $e \cap x = d_{M \cap \lambda} \cap x$.
  Now 
  \[
  H(\theta) \models \mbox{``}\exists z \in \power_\kappa \lambda (z \supseteq x 
  \text{ and } e \cap x = d_z \cap x)\mbox{"},
  \] 
  as witnessed by $M \cap \lambda$. By elementarity, 
  $N$ satisfies this statement as well. Since this holds for all $x \in \power_\kappa \lambda \cap N$, 
  we have 
  \[
  N \models \mbox{``}\forall x \in \power_\kappa \lambda ~ \exists z \in \power_\kappa \lambda 
  (z \supseteq x \text{ and } e \cap x = d_z \cap x)\mbox{"}.
  \] 
  Again by elementarity, this statement holds 
  in $H(\theta)$. But this is precisely the assertion that $e$ is a cofinal branch of $D$, as 
  desired.
\end{proof}

We therefore obtain the following corollary, proving half of Theorem A.

\begin{corollary}
  Suppose that $\mu \leq \kappa$ are regular uncountable cardinals. Then the following are equivalent:
  \bce[(i)]
    \item $\SP(\mu, \kappa, \geq \kappa)$;
    \item $\AGP(\mu, \kappa, \geq \kappa)$.
  \ece
\end{corollary}

In analogy with the principle $\TP^-(\kappa, \lambda)$ from subsection 
\ref{fontanella_matet_subsection}, we will also be interested in weakenings of the 
slender tree property and almost guessing property 
in which ``club" and ``stationary" are replaced by ``strong club" and ``weakly stationary".
Unlike the situation with $\TP^-$ and $\TP$ however, we do not know whether these 
weak versions are equivalent to their seemingly stronger relatives.

\begin{definition}
  Let $\mu \leq \kappa \leq \lambda$ be uncountable cardinals, with $\kappa$ regular, 
  and let $\mc Y \subseteq \power_\kappa \lambda$ be weakly stationary.
  \bce[(i)]
    \item We say that a $(\kappa, \lambda)$-list $D = \langle d_x \mid x \in \power_\kappa \lambda 
    \rangle$ is \emph{strongly $\mu$-$\mc Y$-slender} if for a sufficiently large regular cardinal $\theta$ 
    there is a strong club $C \subseteq \power_\kappa H(\theta)$ such that, for all $M \in C$ and 
    all $y \in M \cap \power_\mu \lambda$, if $M \cap \lambda \in \mc Y$, then there is $e \in M$ such that we have $d_{M \cap \lambda} \cap y = e \cap y$.
    \item The \emph{weak $(\mu, \kappa, \lambda)$-slender tree property on $\mc Y$}, denoted $\wSP_{\mc Y}(\mu, \kappa, 
    \lambda)$, is the assertion that every strongly $\mu$-$\mc Y$-slender $(\kappa, \lambda)$-list has a 
    cofinal branch.
    \item For a regular cardinal $\theta \geq \kappa$ and a weakly stationary $\mc Y' \subseteq \power_\kappa H(\theta)$, the principle $\wAGP_{\mc Y'}(\mu, \kappa, \theta)$ is 
    the assertion that for every cofinal $S \subseteq \power_\kappa H(\theta)$ and every 
    $x \in H(\theta)$, the set of $M \in \power_\kappa H(\theta) \cap \mc Y'$ such that 
    $(M,x)$ is almost $\mu$-guessed by $S$ is weakly stationary in $\power_\kappa H(\theta)$.
  \ece
\end{definition}

Again, the notational conventions from Remarks \ref{tp_convention_remark} and \ref{gmp_convention_remark} 
apply to the principles $\wSP_{\mc Y}(\ldots)$ and $\wAGP_{\mc Y}(\ldots)$, respectively.

\begin{remark} \label{strongly_slender_remark}
  Note that our formulation of strongly $\mu$-slender differs from the formulation of 
  $\mu$-slender in that the conclusion does not require $d_{M \cap \lambda} \cap y \in M$ 
  but rather the existence of $e \in M$ such that $d_{M \cap \lambda} \cap y = e \cap y$. 
  These are clearly equivalent if $M$ is closed under finite intersections, but since 
  we cannot assume that all elements of a strong club are closed under intersections (unlike 
  the situation with clubs), this seems like the more appropriate definition.
\end{remark}

\begin{remark} \label{wsp_tp_remark}
  Just as in Fact \ref{thin_slender_fact}, it can be shown that a thin $(\kappa, \lambda)$-list 
  is in fact \emph{strongly} $\kappa$-slender, so $\wSP(\kappa, \kappa, \lambda)$ is enough to 
  imply $\TP(\kappa, \lambda)$.
\end{remark}

The straightforward analogues of Theorems \ref{slender_guessing_thm_1} and \ref{slender_guessing_thm_2} 
hold for $\wSP$ and $\wAGP$. The proofs are essentially identical, so we omit them, noting 
only that, in the proof of the analogue of Theorem \ref{slender_guessing_thm_1}, 
when verifying that the constructed $(\kappa, H(\theta))$-list is 
strongly $\mu$-$\mc Y$-slender, it is important that the conclusion in the definition 
of strongly $\mu$-slender is weakened from the conclusion of $\mu$-slender, as 
discussed in Remark \ref{strongly_slender_remark}. We therefore 
obtain the following corollary, completing the proof of Theorem A.

\begin{corollary}
  Suppose that $\mu \leq \kappa$ are regular uncountable cardinals. Then the following are equivalent:
  \bce[(i)]
    \item $\wSP(\mu, \kappa, \geq \kappa)$;
    \item $\wAGP(\mu, \kappa, \geq \kappa)$.
  \ece
\end{corollary}

\subsection{Guessing models for small sets}
In this subsection, we examine weakenings of $\GMP$ in which we only require models to 
be guessing for sets of some fixed small cardinality. 
We will see that such principles do not require the full power of $\ISP(\ldots)$ but 
in fact follow already from an appropriate instance of $\SP(\ldots)$. At the same 
time, these principles are strong enough to imply statements such as the failure of the 
weak Kurepa Hypothesis.
For concreteness, we focus on guessing models for sets of size $\omega_1$, but it will be 
evident how to adjust the results for other values of the relevant parameters.

Let us say that $M$ is a \emph{guessing model for sets of size} $\omega_1$ if for every $z\in M$ with $\vert z\vert =\omega_1$, if $d\sub z$ is $(\omega_1, M)$-approximated, then it is $M$-guessed.
We first show that if $M$ is a guessing model for $z$ for some set $z\in M$, 
then $M$ is a guessing model for $y$ for all sets $y$ in $M$ 
which have the same size as $z$. Therefore to show that $$\set{M\prec H(\theta)}{\vert M\vert <\omega_2 \mbox{ and } M \mbox{ is a guessing model for sets of size }\omega_1}$$ is stationary, it is enough to show that $\mathcal{G}^{z}_{\omega_2}H(\theta)$ is stationary for some $z\in H(\theta)$ of size $\omega_1$.

\begin{lemma} \label{bijection_lemma}
Let $M\prec H(\theta)$ and $z\in M$. If $M$ is a guessing model for $z$, then $M$ is a guessing model for $y$ for every $y\in M$ such that $\vert z\vert =\vert y\vert $.
\end{lemma}

\begin{proof}
Let $M\prec H(\theta)$ be a guessing model for $z$ and let $y\in M$ have the same size as $z$. This means that there is in $M$ a bijection $f: y \to z $. Let $d\sub y$ be $M$-approximated. Then $f'' d$ is a subset of $z$, which is also $M$-approximated. To see this, let $a\sub z$ be a countable and in $M$. Then $a\cap f'' d\in M$, since $f^{-1}{} ''( a\cap f''d)=f^{-1}{} ''a \cap d$ is in $M$. Therefore there is $e\in M$, $e\sub z$ such that $e\cap M= f'' d\cap M$ and hence $f^{-1} {}'' e \cap M=d \cap M$.
\end{proof}

\begin{lemma} \label{small_guessing_lemma}
Assume that $\SP(\omega_1,\omega_2, H(\theta))$ holds. Then for all $z\in H(\theta)$ with $\vert z\vert =\omega_1$, the set $\mathcal{G}^z_{\omega_2}H(\theta)$ is stationary in $\power_{\omega_2}H(\theta)$.
\end{lemma}

\begin{proof}
Assume for sake of contradiction that $z\in H(\theta)$, $\vert z\vert =\omega_1$, and $\mathcal{G}^z_{\omega_2}H(\theta)$ is nonstationary. This means that there is a club $C\sub \power_{\omega_2}H(\theta)$ such that for all $M\in C$ there is $d_M\sub z \in M$ such that $d_M$ is $(\omega_1, M)$-approximated but not $M$-guessed. Note that we can assume that $z$ is also a subset of $M$ for all $M\in C$ since $z$ has size $\omega_1$. For $N \in \power_{\omega_2}H(\theta) \setminus C$, choose $M \in C$ such that $N \subseteq M$ and 
let $d_N := d_M \cap N$. 

Let $D=\langle d_M \mid M \in \power_\kappa H(\theta) \rangle$. Again as in the proof of Proposition \ref{guessing_isp_prop_1}, it follows that $D$ is an $\omega_1$-slender $(\omega_2, H(\theta))$-list. Therefore, by  $\SP(\omega_1,\omega_2, H(\theta))$ there is $d\sub H(\theta)$ such that for all $M\in C$ there is $N\supseteq M$ such that
\beq\label{eq:ST1}
d\cap M=d_N\cap M.
\eeq

By our definition of $D$, we can always choose such an $N$ to be in $C$.
By (\ref{eq:ST1}), $d$ is a subset of $z$ and in particular it is in $H(\theta)$. Let $M$ in $C$ be such that $d\in M$. Then there is $N\supseteq M$ such that $N \in C$ and $d \cap M = d_N \cap M$. We then have

\beq\label{eq:ST2}
d=d\cap M=d_N\cap M=d_N=d_N\cap N.
\eeq

The first equality above holds because $d \subseteq z \subseteq M$, and the last two equalities hold because $d_N \subseteq z \subseteq M \subseteq N$. But then $d \in N$ and $d \cap N = d_N \cap N$, contradicting the assumption that $d_N$ is not $N$-guessed.
\end{proof}

Putting together a couple of results, we can actually show that 
$\SP(\omega_1, \omega_2, H(\omega_2))$ is sufficient to obtain the conclusion of the previous 
lemma.

\begin{corollary}
Assume that $\SP(\omega_1,\omega_2,H(\omega_2))$ holds. Then for every 
regular $\theta \geq \omega_2$, the set 
\[
\set{M\prec H(\theta)}{\vert M\vert <\omega_2 \mbox{ and } M \mbox{ is a guessing model for sets of size }\omega_1}
\] 
is stationary in $\power_{\omega_2}H(\theta)$.
\end{corollary}

\begin{proof}
  By Lemma \ref{bijection_lemma}, it is enough to show that, for every regular 
  $\theta \geq \omega_2$, $\mc G^{\omega_1}_{\omega_2} H(\theta)$ is stationary in 
 $\power_{\omega_2} H(\theta)$. By Clause \ref{transfer_clause_1} of Proposition 
 \ref{transfer_prop}, this is equivalent to the assertion that $\mc G^{\omega_1}_{\omega_2} 
 H(\omega_2)$ is stationary in $\power_{\omega_2} H(\omega_2)$. By Lemma 
 \ref{small_guessing_lemma}, this follows from the hypothesis of 
 $\SP(\omega_1, \omega_2,H(\omega_2))$.
\end{proof}

Since every element of $H(\omega_2)$ has cardinality at most $\omega_1$, it follows 
that the principle $\GMP(\omega_1, \omega_2, H(\omega_2))$ is equivalent to the conclusion of the previous 
corollary. In particular, as a special case of Theorem \ref{th:SP} below, we see 
that the existence of guessing models for sets of size $\omega_1$ is sufficient to imply 
the nonexistence of weak Kurepa trees.

\section{Subadditive functions and weak guessing properties} \label{subadditive_section}

\begin{definition}
  Suppose that $\chi$ and $\lambda$ are infinite cardinals and $c:[\lambda]^2 \rightarrow \chi$. 
  We say that $c$ is \emph{subadditive} if, for all $\alpha < \beta < \gamma < \lambda$, the following 
  two triangle inequalities hold:
  \bce[(i)]
    \item $c(\alpha, \gamma) \leq \max\{c(\alpha, \beta), c(\beta, \gamma)\}$;
    \item $c(\alpha, \beta) \leq \max\{c(\alpha, \gamma), c(\beta, \gamma)\}$.
  \ece
  We say that $c$ is \emph{strongly unbounded} if, for every unbounded $A \subseteq \lambda$, 
  $c'' [A]^2$ is unbounded in $\chi$.
\end{definition}

In \cite[Theorem 10.3 and Proposition 6.1]{narrow_systems}, it is shown that 
$\GMP(\omega_1, \omega_2, \geq \omega_2)$ implies that, for every regular $\lambda \geq \omega_2$, 
there are no subadditive, strongly unbounded functions $c:[\lambda]^2 \rightarrow \omega$. We now prove 
a generalization and strengthening of this result by proving that the nonexistence of subadditive, 
strongly unbounded functions follows from relevant instances of $\wAGP(\ldots)$. Together with 
Corollary \ref{square_failure_cor}, this yields clause (2) of Theorem C. In Section 
\ref{mitchell_section}, we will see that the hypotheses of the following theorem hold, for instance, 
after forcing with $\M(\mu, \kappa)$ when $\kappa$ is strongly compact, where $\mu$ is a regular infinite 
cardinal and $\M(\mu, \kappa)$ is the Mitchell forcing that collapses $\kappa$ to be $\mu^{++}$.

\begin{theorem} \label{subadditive_function_theorem}
  Suppose that the following hypotheses hold:
  \bce
    \item $\chi < \chi^+ < \kappa \leq \lambda$ are infinite cardinals, with 
    $\kappa$ regular and $\cf(\lambda) \geq \kappa$;
    \item $\mc Y := \{M \in \power_\kappa H(\lambda^+) \mid \cf(\sup(M \cap \lambda)) > \chi\}$;
    \item $\wAGP_{\mc Y}(\kappa, \kappa, \lambda^+)$ holds.
  \ece
  Then there are no subadditive, strongly unbounded functions $c:[\lambda]^2 \rightarrow \chi$.
\end{theorem}

\begin{proof}
  Fix a subadditive function $c:[\lambda]^2 \rightarrow \chi$. We will find an unbounded 
  $A \subseteq \lambda$ such that $c''[A]^2$ is bounded below $\chi$. 
  
  For each $\beta < \lambda$ and each $i < \chi$, let $c_{\beta, i}:\beta \rightarrow \chi$ 
  be defined by setting 
  \[
    c_{\beta, i}(\alpha) := \begin{cases}
      c(\alpha, \beta) & \text{if } c(\alpha, \beta) \geq i \\
      i & \text{if } c(\alpha, \beta) < i.
    \end{cases}
  \]
  Note that the subadditivity of $c$ implies that, for all $\alpha < \beta < \lambda$ and all 
  $i < \chi$, if $c(\alpha, \beta) \leq i$, then $c_{\beta, i} \restriction \alpha = c_{\alpha, i}$.
  We will think of functions such as $c_{\beta, i}$ as subsets of $\lambda \times \chi$ in the 
  natural way.
  
  Let $S$ be the set of $N \prec H(\lambda^+)$ such that $\vert N\vert  < \kappa$, $\chi \subseteq N$,
  $c \in N$, and $\cf(\sup(N \cap \lambda)) > \chi$. Since $\wAGP_{\mc Y}(\kappa, \kappa, H(\lambda^+))$ 
  holds, there are weakly stationarily many $M \in \mc Y$ such that $(M, \lambda \times \chi)$ is almost 
  $\kappa$-guessed by $S$. In particular, we can find such an $M \in \mc Y$ with the following additional 
  properties:
  \begin{itemize}
    \item $c, \lambda \times \omega \in M$;
    \item for all $\beta \in M \cap \lambda$ and all $i < \theta$, we have $c_{\beta, i} \in M$;
    \item for all $z \in M \cap \power_\kappa(\lambda \times \chi)$, we have 
    $\sup\{\alpha < \lambda \mid \exists i < \chi [(\alpha, i) \in z]\} \in M$.
  \end{itemize}
  Let $\gamma := \sup(M \cap \lambda)$. Since $M \in \mc Y$, we have 
  $\cf(\gamma) > \chi$. We can therefore fix an $i_0 < \chi$ and an unbounded $B \subseteq M \cap \lambda$ 
  such that, for all $\beta \in B$, we have $c(\beta, \gamma) = i_0$.
  
  \begin{claim}
    $c_{\gamma, i_0}$ is $(\kappa, M)$-approximated.
  \end{claim}
  
  \begin{proof}
    Fix a set $z \in M \cap \power_\kappa(\lambda \times \chi)$. We must find $e \in M$ such that
    $c_{\gamma, i_0} \cap z = e \cap z$. We have $\sup\{\alpha < \lambda \mid \exists 
    i < \chi [(\alpha, i) \in z]\} \in M$, so we can find $\beta \in B$ such that 
    $z \subseteq \beta \times \chi$. Since $c(\beta, \gamma) = i_0$, we have 
    $c_{\gamma, i_0} \restriction \beta = c_{\beta, i_0}$, and hence $c_{\gamma, i_0} \cap z = 
    c_{\beta, i_0} \cap z$. Moreover, we have $c_{\beta, i_0} \in M$, so $c_{\beta, i_0}$ is as desired.
  \end{proof}
  
  Since $(M, \lambda \times \chi)$ is almost $\kappa$-guessed by $S$, we can find $N \in S$ 
  such that $N \subseteq M$ and $c_{\gamma, i_0}$ is $N$-guessed, 
  i.e., there is $e \in N$ such that $c_{\gamma, i_0} \cap N = e \cap N$. Note that 
  $c_{\gamma, i_0} \cap N$ is a function from $N \cap \lambda$ to $\chi$. By elementarity and 
  the fact that $\chi + 1 \subseteq N$, it follows that $e$ is a function from $\lambda$ to $\chi$.
  Let $\delta := \sup(N \cap \lambda)$. Since $N \in S$, we have $\cf(\delta) > \chi$. 
  We can therefore find $i_1 \in [i_0, \chi)$ such that $A_0 := \{\alpha \in N \cap \lambda \mid e(\alpha) 
  \leq i_1\}$ is unbounded in $\delta$. Let $A := \{\alpha < \lambda \mid e(\alpha) \leq i_1\}$. All 
  of the parameters needed to define $A$ are in $N$, so $A \in N$. Moreover, $N \models$ ``$A \text{ is 
  unbounded in } \lambda$'', so, by elementarity, $A$ is in fact unbounded in $\lambda$. 
  We will therefore be finished if we show that $c(\alpha, \beta) \leq i_1$ for all $\alpha < \beta$ 
  in $A$. By elementarity, it suffices to show that $c(\alpha, \beta) \leq i_1$ for all 
  $\alpha < \beta$ in $N \cap A$.
  
  To this end, fix $\alpha < \beta$ in $N \cap A$. By definition of $A$, we know that $e(\alpha), 
  e(\beta) \leq i_1$. Since $N \subseteq M$, $e \cap N = c_{\gamma, i_0} \cap N$, and 
  $i_1 \geq i_0$, it follows that $\max\{c(\alpha, \gamma), c(\beta, \gamma)\} \leq i_1$. By the 
  subadditivity of $c$, we can conclude that $c(\alpha, \beta) \leq i_1$, finishing the proof.
\end{proof}

\begin{corollary} \label{square_failure_cor}
  Suppose that $\omega_2 \leq \kappa \leq \lambda$ are regular cardinals, 
  \[
  \mc Y := \{M \in 
  \power_\kappa H(\lambda^+) \mid \cf(\sup(M \cap \lambda)) > \omega\},
  \] and 
  $\wAGP_{\mc Y}(\kappa, \kappa, \lambda^+)$ holds. 
  Then $\square(\lambda)$ fails.
\end{corollary}

\begin{proof}
  By the combination of \cite[Proposition 6.5]{narrow_systems} and \cite[Theorem 3.4]{lh_lucke}, 
  $\square(\lambda)$ implies the existence of a subadditive, strongly unbounded function 
  $c:[\lambda]^2 \rightarrow \omega$. The corollary now follows from Theorem 
  \ref{subadditive_function_theorem}.
\end{proof}

\begin{remark}
  Note that, by monotonicity, the hypothesis of Corollary \ref{square_failure_cor} 
  for a particular regular $\lambda \geq \kappa$ implies the hypothesis for 
  all regular $\lambda'$ in the interval $[\kappa,\lambda]$.
  In Theorem \ref{strongly_compact_thm_2} (cf.\ Remark \ref{wsp_remark}), 
  starting from a model of $\ZFC$ with a strongly compact cardinal, we 
  will obtain a model in which the hypothesis of Corollary \ref{square_failure_cor} 
  holds for $\kappa = \omega_2$, simultaneously for all regular $\lambda \geq 
  \omega_2$ (in fact, we will obtain the stronger $\wAGP_{\mc Y}(\omega_1, 
  \omega_2, \lambda)$ for all regular $\lambda \geq \omega_2$, where $\mc Y$ is as 
  above). Corollary \ref{square_failure_cor} shows that the consistency of 
  significant large cardinals is indeed necessary for obtaining the consistency of 
  such guessing principles. For example, by \cite[Theorem 0.1]{stacking_mice}, if 
  $\mu \geq \omega_3$ is a regular cardinal such that $\gamma^\omega < \mu$ for 
  all $\gamma < \mu$ and both $\square(\mu)$ and $\square(\mu^+)$ fail, then 
  there is a sharp for a proper class model with a proper class of strong 
  cardinals and a proper class of Woodin cardinals. The same conclusion therefore 
  follows from, for instance, $\wAGP_{\mc Y}(\kappa, \kappa, (\kappa^\omega)^{+3})$, 
  where $\kappa \geq \omega_2$ is regular and $\mc Y$ is as above.
\end{remark}

\section{Preservation lemmas} \label{preservation_section}

In this section, we prove a variety of preservation lemmas that will be used in our 
consistency results in the remaining sections of the paper.

\begin{lemma}
  Suppose that $\mu < \kappa$ are regular uncountable cardinals, $\Lambda$ is a 
  $\kappa$-directed poset, $T$ is a $\Lambda$-tree, $\P$ is a $\mu$-c.c.\ forcing notion, 
  and 
  \[
    \Vdash_{\P}\mbox{``}\text{there is a cofinal branch through } T\mbox{''}.
  \] 
  Then there is a cofinal branch through $T$ in $V$.
\end{lemma}

\begin{proof}
  Let $\dot{b}$ be a $\P$-name for a cofinal branch through $T$. For each $u \in \Lambda$, 
  let 
  \[
    T'_u := \{s \in T_u \mid \exists p \in \P ~ [p \Vdash_{\P} \mbox{``}\dot{b}(u) = s\mbox{''}]\}.
  \]
  Note that, if $u <_\Lambda v$, $t \in T_v$, $s = t \restriction u$, $p \in \P$, and 
  $p \Vdash_{\P}$``$\dot{b}(v) = t$'', then we also have $p \Vdash_{\P}$``$\dot{b}(u) = s$''. 
  It follows that $T' := (\langle T'_u \mid u \in \Lambda \rangle, <_{T'})$ is a subtree 
  of $T$, where $<_{T'}$ is the restriction of $<_T$ to $\bigcup_{u \in \Lambda} T'_u$. 
  Also, since $\P$ has the $\mu$-c.c., we know that $\vert T'_u\vert  < \mu$ for all $u \in \Lambda$. 
  In particular, since $\mu < \kappa$ and $\Lambda$ is $\kappa$-directed, $T'$ is 
  very thin. By Lemma \ref{very_thin_branch_lemma}, $T'$ has a cofinal branch, which is 
  plainly a cofinal branch through $T$, as well.
\end{proof}

\begin{lemma} \label{closed_preservation_lemma}
Let $\xi$ be a cardinal and $\mu < \kappa$ be regular cardinals such that $2^\mu\ge \xi$ and $2^{<\mu}<\kappa$. Let $\Q$ be a $\mu^+$-closed forcing and $\P$ be $\mu^+$-cc. Assume that $\Lambda$ is a $\kappa$-directed poset in $V$. If $T$ is a $\Lambda$-tree with width at most $\xi$ in $V^{\P}$, then forcing with $\Q$ over $V^{\P}$ does not add a cofinal branch through $T$.
\end{lemma}

\begin{proof}

Let $p\in \P$ be a condition which forces that $\dot T$ is a $\Lambda$-tree of width at most $\xi$, and assume further that for some  $q\in \Q$, $(p,q)$ forces that $\dot{b}$ is a cofinal branch through $\dot T$ which is not in $V^{\P}$ (we view $\dot{b}$ as a $\P \x \Q$-name).

First we prove the following auxiliary claim.

\begin{claim}\label{Cl:ccc-closed}
Assume that $q'_0$ and $q'_1$ are conditions in $\Q$ which extend $q$, and let $x'\in \Lambda$. Then there are a maximal antichain $Y$ in $\P$ below $p$, conditions $q_0\le_{\Q} q'_0$, $q_1\le_{\Q} q'_1$, and $x>_\Lambda x'$ such that whenever $p'\in Y$, then $(p',q_0)$ and $(p',q_1)$ force contradictory information about $\dot{b}$ on level $x$; i.e.\ there are $\P$-names $\dot{t}_0$ and $\dot{t}_1$ such that $p'\Vdash_{\P} \dot{t}_0\neq\dot{t}_1\in \dot{T}_{x}$ and 
$$(p',q_0)\Vdash_{\P \times \Q} \dot{b}(x) = \dot{t}_0 \mbox{ and }(p',q_1)\Vdash_{\P \times \Q} \dot{b}(x) = \dot{t}_1.$$
\end{claim}

\begin{proof}
Let $q'_0, q'_1\le_{\Q} q$ and $x'\in \Lambda$ be given. We will construct by induction a maximal antichain $Y=\set{p'_i\in \P}{i<\gamma}$ below $p$, decreasing sequences $\seq{q^i_0}{i<\gamma}$ and $\seq{q^i_1}{i<\gamma}$ of conditions in $\Q$ and a $<_\Lambda$-increasing sequence $\seq{x_i}{i<\gamma}$, for some $\gamma<\mu^+$.

Assume  $\delta<\mu^+$ and we already constructed $\set{p'_i\in \P}{i<\delta}$, $\seq{q^i_0}{i<\delta}$, $\seq{q^i_1}{i<\delta}$ and $\seq{x_i}{i<\delta}$. 
Suppose that $\set{p'_i\in \P}{i<\delta}$ is not a maximal antichain in $\P$ below $p$. We distinguish two cases: (A) $\delta$ is a successor ordinal, and (B) $\delta$ is a limit ordinal.

Case (A).  Suppose $\delta=\delta'+1$. Since $\set{p'_i}{i<\delta}$ is not a maximal antichain in $\P$ below $p$, there is $p^*\in \P$ such that $p^* \le_{\P} p$ and $p^*$ is incompatible with $p'_i$ for all $i<\delta$. Now, consider the condition $(p^*, q^0_{\delta'})$ and $x_{\delta'}$. Since $\dot{b}$ is forced by $(p,q)$ to be a cofinal branch through $\dot{T}$ which is not in $V^{\P}$, there are $x'>_\Lambda x_{\delta'}$, $p'\le_{\P} p^*$,  and $r_0$ and $r_1$ both extending $q^0_{\delta'}$ in $\Q$, such that $(p',r_ 0)\Vdash_{\P \times \Q} \dot{b}(x')= \dot {t}_0$ and $(p',r_1)\Vdash_{\P \times\Q} \dot{b}(x')= \dot{t}_1$, where $\dot{t}_0$ and $\dot{t}_1$ are forced by $p'$ to be distinct elements of $\dot{T}_{x'}$. Moreover, let $(p'',r_2)\le (p',q^1_{\delta'})$ be such that it decides $\dot{b}$ on level $x'$; i.e\ there is a $P$-name $\dot{t}$ such that   $(p'',r_2) \Vdash_{\P \times \Q} \dot{b}(x')= \dot t$. Since $p''\le p'$ and $p'\Vdash_{\P} \dot{t}_0\neq\dot {t}_1$, $p''$ has to force that $\dot t \neq  \dot {t}_0$ or $\dot t\neq \dot{t}_1$. If $p''$ forces that $\dot t \neq  \dot {t}_0$ then define $q^0_\delta=r_0$, if not then define $q^0_\delta=r_1$. Let $p'_\delta=p''$, $q^1_\delta=r_2$ and $x_\delta=x'$.

Case (B). If $\delta$ is a limit ordinal, then we first take $q^*_0$ and $q^*_1$ to be some lower bounds of $\seq{q^i_0}{i<\delta}$ and $\seq{q_1^i}{i<\delta}$ respectively and $x^*$ be an upper bound of $\seq{x_i}{i<\delta}$. Then we proceed as in the successor step, using $q^*_0$, $q^*_1$ and $x^*$ instead of $q^0_{\delta'}$, $q^1_{\delta'}$ and $x_{\delta'}$, respectively.

By the $\kappa^+$-cc of $\P$, the inductive construction must stop at some $\delta < \kappa^+$ in the sense that $\set{p'_i}{i<\delta}$ is a maximal antichain in $\P$ below $p$. Let $Y=\set{p'_i}{i<\delta}$, let $q_ 0$ and $q_1$ be some lower bounds of $\seq{q^0_i}{i<\delta}$ and $\seq{q^1_i}{i<\delta}$, respectively, and finally let $x$ be an upper bound of $\set{x_i}{i<\delta}$.

It is easy to verify that for every $p'\in Y$, $(p',q_0)$ and $(p',q_1)$ force contradictory information about $\dot{b}$ on level $x$. This follows immediately from the construction of $Y$, $q_0$ and $q_1$.
\end{proof}

Iteratively using Claim \ref{Cl:ccc-closed}, one can build a labeled full binary tree $\mathscr{T}$ in $2^{<\mu}$ such that for every $s\in 2^{<\mu}$ there are $q_s\in \Q$, $x_s\in \Lambda$ and $Y_s \subseteq \P$ such that the following hold:

\bce
\item For all $s \in 2^{<\mu}$, $q_s \leq q$ and $Y_s$ is a maximal antichain in $\P$ below $p$.
\item For all $s\in 2^{<\mu}$ and $p'\in Y_s$, $(p',q_{s^\frown 0})$ and $(p',q_{s^\frown 1})$  force contradictory information about $\dot{b}$ on level $x_s$. 
\item For all $f\in 2^\mu$, the conditions $\seq{q_{f\rest\alpha}}{\alpha<\mu}$ are decreasing in $\le_Q$.

\item For all $f\in 2^\mu$, the sets $\seq{x_{f\rest\alpha}}{\alpha<\mu}$ are increasing in $<_\Lambda$.
\ece
Suppose for the moment that $\mathscr{T}$ has been constructed; then we can conclude the proof as follows. Since $2^{<\mu} < \kappa$ and $\Lambda$ is $\kappa$-directed, we can fix an $x \in \Lambda$ such that $x_s \leq_{\Lambda} x$ for all $s \in 2^{<\mu}$. Let $G$ be a $\P$-generic filter over $V$ such that $p\in \P$. Next, still in $V$, let us fix lower bounds $q_f$ of $\seq{q_{f\rest\alpha}}{\alpha<\mu}$ for all $f\in 2^\mu$ (they exist since $\Q$ is $\mu^+$-closed in $V$). In $V[G$],  for each $f$, we find $q'_f\le q_f$ such that $q'_f$ decides $\dot{b}(x)$. We claim that, for each $f\neq g\in (2^\mu)^V$, $q'_f$ and $q'_g$ force contradictory information about $\dot{b}$ on level $x$. To see this, assume that $f\neq g\in (2^\mu)^V$. Let $s=f\cap g$, and without loss of generality let $f$ extend $s^\frown 0$ and $g$ extend $s^\frown 1$. By the properties of the tree $\mathscr{T}$, there is $p'\in Y_s$ such that $p'\in G$ and $(p',q_{s^\frown 0})$ and $(p',q_{s^\frown 1})$ force contradictory information about $\dot{b}$ on level $x_s $. Since $q'_f\le q_{s^\frown 0}$ and $q'_g\le q_{s^\frown 1}$, $(p',q'_f)$ and $(p',q'_g)$ force contradictory information about $\dot{b}$ on level $x_s$, and hence also on level $x>_\Lambda x_s$. However, this is a contradiction since $2^\mu\ge\xi$ and there are only $<\xi$ many nodes on level $x$ in $T$; i.e., $<\xi$ many possible values of $\dot{b}(x)$.

To finish the argument it is enough to construct the tree $\mathscr{T}$. The construction will proceed by induction on the length of $s$; during the construction we will also construct an auxiliary increasing sequence $\seq{x_\alpha}{\alpha<\mu}$ of elements of $\Lambda$.

For $\alpha=0$, let $q_\emptyset=q$ and let $x_0 = x_\emptyset$ be an arbitrary element of $\Lambda$.

Assume now that we have constructed $\mathscr{T}\rest\alpha$. If $\alpha$ is limit, then for $s\in 2^\alpha$, let $q_s$ be a lower bound of $\seq{q_{s\rest\beta}}{\beta<\alpha}$ and  $x_\alpha$  an upper bound of $\set{x_t\in\Lambda}{t\in 2^{<\alpha}}$.  Note that $x_\alpha$ exists since we assume that $\Lambda$ is $\kappa$-directed and $2^{<\mu}<\kappa$.

If $\alpha$ is equal to $\beta+1$, and $s\in 2^\beta$, then the existence of $q_{s^\frown 0}$, $q_{s^\frown 1}\in \Q$, $x_s\in \Lambda$ and $Y_s$ follows from Claim \ref{Cl:ccc-closed} starting with $q'_0=q'_1=q_s$ and $x'=x_\alpha$, and defining $Y_s=Y$,  $q_{s^\frown 0}=q_0$, $q_{s^\frown 1}=q_1$ and $x_s=x$.
The properties  (1)--(4) of the labeled tree $\mathscr{T}$ now follow immediately from the construction. 
\end{proof}

\begin{corollary}\label{Cor:ccc-closed}
Let $\xi$ be a cardinal and $\mu < \kappa\le\lambda$ be regular cardinals such that $2^\mu\ge \xi$ and $2^{<\mu}<\kappa$. Let $\Q$ be $\mu^+$-closed forcing and $\P$ be $\mu^+$-cc. If $T$ is a $\power_\kappa\lambda$-tree with width at most $\xi$ in $V^{\P}$, then forcing with $\Q$ over $V^{\P}$ does not add a cofinal branch through $T$.
\end{corollary}

\begin{proof}
Note that, since $\P$ is $\mu^+$-cc, $(\power_\kappa\lambda)^V$ is $\subseteq$-cofinal in $(\power_\kappa\lambda)^{V^{\P}}$. Hence the proof follows immediately from the previous theorem. If $T$ is a $(\power_\kappa\lambda)^{V^{\P}}$-tree in $V^{\P}$ with a new cofinal branch in $V^{\P \times\Q}$, then $T'=\set{T_x}{x\in (\power_\kappa\lambda)^V}$ also has a new cofinal branch in $V^{\P \times \Q}$, which is a contradiction to the previous theorem.
\end{proof}

We next recall the \emph{covering} and \emph{approximation} properties, introduced by 
Hamkins (cf.\ \cite{hamkins_approximation}).

\begin{definition}
  Suppose that $V \subseteq W$ are transitive models of $\ZFC$ and $\mu$ is a regular uncountable cardinal.
  \begin{enumerate}
    \item $(V,W)$ satisfies the \emph{$\mu$-covering property} if, for every $x \in W$ such that 
    $x \subseteq V$ and $\vert x\vert ^W < \mu$, there is $y \in V$ such that $\vert y\vert ^V < \mu$ and $x \subseteq y$.
    \item $(V,W)$ satisfies the \emph{$\mu$-approximation property} if, for all $x \in W$ such that 
    $x \subseteq V$ and $x \cap z \in V$ for all $z \in V$ with $\vert z\vert  < \mu$, we in fact have $x \in V$.
  \end{enumerate}
  A poset $\P$ has the \emph{$\mu$-covering property} (resp.\ \emph{$\mu$-approximation 
  property}) if, for every $V$-generic filter $G \subseteq \P$, the pair $(V,V[G])$ has the 
  $\mu$-covering property (resp.\ $\mu$-approximation property).
\end{definition}

\begin{lemma} \label{approximation_branch_lemma}
  Suppose that $V \subseteq W$ are transitive models of $\ZFC$, $\kappa$ is a regular uncountable cardinal, 
  and $(V,W)$ has the $\kappa$-approximation property. Suppose also that, in $V$, $\Lambda$ is a 
  $\kappa$-directed partial order and $T$ is a $\Lambda$-tree. Then every cofinal branch through $T$ 
  in $W$ is already in $V$.
\end{lemma}

\begin{proof}
  Suppose that $b \in W$ is a cofinal branch through $T$. We will show that $b \in V$. Towards an 
  application of the $\kappa$-approximation property, fix a set $z \in V$ with $\vert z\vert  < \kappa$. 
  We want to show that $b \cap z \in V$, so we may as well assume that 
  $z \subseteq \bigcup \{\{u\} \times T_u \mid u \in \Lambda\}$, since $b$ is also a subset of this set.
  Let 
  \[
    W := \{u \in \Lambda \mid \exists t \in T_u [(u,t) \in z]\}.
  \]
  Then $\vert W\vert  < \kappa$ and $\Lambda$ is $\kappa$-directed, we can find $v \in \Lambda$ such that 
  $u <_{\Lambda} v$ for all $u \in W$. Let $s := b(v)$. Then $b \cap z = \{(u,t) \in z \mid t <_T s\}$. 
  All of the parameters in the right-hand side of this equation are in $V$, so $b \cap z \in V$. 
  Since $z$ was arbitrary, the $\kappa$-approximation property implies that $b \in V$.
\end{proof}

\begin{lemma} \label{suslin_preservation_lemma}
  Suppose that $\kappa$ is a regular uncountable cardinal, 
  $\Lambda$ is a $\kappa$-directed partial order, and $T$ is a $\Lambda$-tree. Suppose 
  also that $\bb{P}$ is a $\kappa$-c.c.\ forcing that adds a 
  new cofinal branch to $T$. Then there is a $\kappa$-Suslin tree $S$ such that 
  $\bb{P}$ adds a cofinal branch to $S$.
\end{lemma}

\begin{proof}
  Let $\dot{b}$ be a $\bb{P}$-name for a new cofinal branch through $T$.
  For all $u \in \Lambda$, let $B_u := \{t \in T_u \mid \exists p \in \P ~ (p \Vdash_{\P}
  \mbox{``}\dot{b}(u) = t\mbox{''})\}$. Since $\P$ has the $\kappa$-c.c., each $B_u$ has cardinality 
  less than $\kappa$.
  
  We now recursively define a $<_\Lambda$-increasing sequence $\langle u_\eta \mid \eta 
  < \kappa \rangle$ as follows. First, since $\Vdash_{\P}$``$\dot{b} \notin V$'', we can 
  find $u_0 \in \Lambda$ such that $\vert B_{u_0}\vert  > 1$. Next, suppose that $\eta < \kappa$ and 
  we have defined $u_\eta$. For each $t \in B_{u_\eta}$, again using the fact that 
  $\Vdash_{\P}$``$\dot{b} \notin V$'', we can find $u_{\eta,t} \in \Lambda$ such 
  that $u_\eta <_\Lambda u_{\eta,t}$ and $\vert \{t' \in B_{u_{\eta,t}} \mid t' \restriction 
  u_\eta = t\}\vert  > 1$. Since $\vert B_{u_\eta}\vert  < \kappa$ 
  and $\Lambda$ is $\kappa$-directed, it follows that 
  we can find $u_{\eta+1} \in \Lambda$ such that $u_{\eta,t} \leq_\Lambda u_{\eta+1}$ 
  for all $t \in B_{u_\eta}$. Finally, if $\xi < \kappa$ is a limit ordinal and 
  we have defined $\langle u_\eta \mid \eta < \xi \rangle$, again use the 
  $\kappa$-directedness of $\Lambda$ to find $u_\xi \in 
  \Lambda$ such that $u_\eta \leq_\Lambda u_\xi$ for all $\eta < \xi$.
  
  Now define a tree $S$ by letting the underlying set of $S$ be $\bigcup 
  \{B_{u_\eta} \mid \eta < \kappa\}$ and letting the ordering $<_S$ be the restriction of 
  $<_T$ to $S$. It is immediate that $S$ is a tree of height $\kappa$ and, for 
  all $\eta < \kappa$, the $\eta^{\mathrm{th}}$ level of $S$ is precisely $B_{u_\eta}$. 
  Since $\vert B_{u_\eta}\vert  < \kappa$, it follows that $S$ is a $\kappa$-tree. Moreover, 
  it follows from our construction that, for all $\eta < \xi < \kappa$ and all 
  $t \in B_{u_\eta}$, there are distinct $t_0, t_1 \in B_{u_\xi}$ such that 
  $t_0 \restriction u_\eta = t_1 \restriction u_\eta = t$. Therefore, $S$ is splitting, 
  so to show that $S$ is $\kappa$-Suslin, it suffices to show that it has no antichain 
  of cardinality $\kappa$. To this end, suppose that $\langle t_\alpha \mid \alpha < \kappa 
  \rangle$ is a sequence of elements of $S$. For each $\alpha < \kappa$, fix 
  $\eta_\alpha < \kappa$ such that $t_\alpha \in B_{u_{\eta_\alpha}}$.
  For each $\alpha < \kappa$, fix $p_\alpha \in \P$ such 
  that $p_\alpha \Vdash_{\P}$``$\dot{b}(u_{\eta_\alpha}) = t_\alpha$''. Since 
  $\P$ has the $\kappa$-c.c., we can find $\alpha < \beta < \kappa$ such that $p_\alpha$ 
  and $p_\beta$ are compatible in $\P$. Let $q \in \P$ be a common extension of $p_\alpha$ 
  and $p_\beta$. Then $q\Vdash_{\P}$``$(\dot{b}(u_{\eta_\alpha}) = t_\alpha) \wedge 
  (\dot{b}(u_{\eta_\beta}) = t_\beta)$''. In particular, it follows that 
  $t_\alpha$ and $t_\beta$ are comparable in $S$, so $\langle t_\alpha \mid \alpha < 
  \kappa \rangle$ is not an antichain in $S$.
  
  Finally, the interpretation of $\dot{b}$ in any forcing extension by $\P$ defines a 
  cofinal branch through $S$, namely $\{\dot{b}(u_\eta) \mid \eta < \kappa\}$. 
  Therefore, $\P$ necessarily adds a cofinal branch to $S$.
\end{proof}

\section{Mitchell forcing} \label{mitchell_section}

A variant of Mitchell forcing will be one of our primary tools for proving consistency results. We begin this 
section by briefly reviewing its definition and some of its important properties. Throughout this 
section, let $\mu$ be an infinite cardinal such that $\mu^{<\mu} = \mu$, and let $\delta$ be an ordinal 
with $\cf(\delta) > \mu$. In practice, $\delta$ will typically be (at least) a Mahlo cardinal, but 
will sometimes need to consider $\delta$ of the form $j(\kappa)$, where $j:V \rightarrow M$ is an 
elementary embedding with critical point $\kappa$, in which case $\delta$ is inaccessible in $M$ but 
may not even be a cardinal in $V$.

For this section, let $\P$ denote the forcing $\Add(\mu, \delta)$ for adding $\delta$-many Cohen subsets 
to $\mu$. More precisely, $\P$ consists of all partial functions of size ${<}\mu$ from $\delta$ to 
${^{<\mu}}2$, where, for $q,p \in \P$, we have $q \leq_{\P} p$ if and only if $\dom{q} \supseteq 
\dom{p}$ and, for all $\alpha \in \dom{p}$, $q(\alpha)$ end-extends $p(\alpha)$. 
For $\beta < \delta$, let $\P_\beta$ denote the suborder $\Add(\mu, \beta)$. We can now 
define the variant of the Mitchell forcing that we will use, $\M = \M(\mu, \delta)$, as follows. First, let $A$ be some unbounded subset of $\delta$ with $\min(A) > \mu$. 
Recall that $\acc(A)$, the set of \emph{accumulation points} 
of $A$, is defined to be $\{\alpha \in A \mid \sup(A \cap \alpha) = \alpha\}$, and 
$\nacc(A) := A \setminus \acc(A)$. Also, $\acc^+(A):= \{\alpha < \sup(A) \mid \sup(A \cap \alpha) 
= \alpha\}$. Though our definition of $\M$ will depend on our choice of $A$, we suppress mention of 
$A$ in the notation; if we need to make the choice of $A$ explicit, we will speak of ``the Mitchell 
forcing $\M(\mu, \delta)$ defined using the set $A$". 
Unless otherwise specified, we will always let $A$ be the set of
all inaccessible cardinals in the interval $(\mu, \delta)$, but the definition of $\M$ and the properties 
listed below work for any choice of $A$.  Conditions in $\M$ are all pairs 
$(p^0, p^1)$ such that 
\begin{itemize}
  \item $p^0 \in \bb{P}$;
  \item $p^1$ is a function and $\dom{p^1} \in [\nacc(A)]^{\leq \mu}$;
  \item for all $\alpha \in \dom{p^1}$, $p^0 \restriction \alpha \Vdash_{\P_\alpha}$ 
  ``$p^1(\alpha) \in \Add(\mu^+,1)^{V^{\P_\alpha}}$''.
\end{itemize}
If $(p^0,p^1), (q^0,q^1) \in \M$, then $(q^0,q^1) \leq_\M (p^0,p^1)$ if and only if
\begin{itemize}
  \item $q^0 \leq_{\P} p^0$;
  \item $\dom{q^1} \supseteq \dom{p^1}$;
  \item for all $\alpha \in \dom{p^1}$, $q^0 \restriction \alpha \Vdash_{\P_\alpha}$ 
  ``$q^1(\alpha) \leq p^1(\alpha)$''. 
\end{itemize} 
For $\beta < \kappa$, we let $\M_\beta$ denote the suborder of $\M$ consisting of all conditions 
$(p^0,p^1)$ such that the domains of both $p^0$ and $p^1$ are contained in $\beta$. We will sometimes 
denote $\M$ by $\M_\delta$. 

\begin{remark} \label{mitchell_remark}
The following are some of the key properties of $\M$ (cf.\ 
\cite{abraham}, \cite{mitchell_approximation} for further details and proofs):
\bce
  \item $\M$ is $\mu$-closed and, if $\delta$ is inaccessible, it is $\delta$-Knaster.
  \item $\M$ has the $\mu^+$-covering and $\mu^+$-approximation properties. Together with the previous 
  item, this implies that, if $\delta$ is inaccessible, then forcing with $\M$ preserves all cardinals less 
  than or equal to $\mu^+$ and all cardinals greater than or equal to $\kappa$.
  \item If $\delta$ is inaccessible, then $\Vdash_{\M} $``$2^\mu = \delta = \mu^{++}$''.
  \item There is a projection onto $\M$ from a forcing of the form $\Add(\mu, \delta) \times 
  \Q$, where $\Q$ is $\mu^+$-closed.
  \item \label{quotient_approx_prop} 
  For all inaccessible cardinals $\alpha \in \acc^+(A)$,
  there is a projection from $\M$ to $\M_\alpha$ and, in $V^{\M_\alpha}$, the quotient 
  forcing $\M/\M_\alpha$ has the $\mu^+$-approximation property.
  \item \label{quotient_product_prop} 
  For all inaccessible cardinals $\alpha \in \acc^+(A)$, let $\alpha^\dagger$ denote 
  $\min(A \setminus \alpha+1)$. Then, in $V^{\M_\alpha}$, the quotient 
  forcing $\M/\M_\alpha$ is of the form $\Add(\mu, \alpha^\dagger - \alpha) \ast \dot{\M}^\alpha$, where, 
  in $V^{\M \ast \Add(\mu, \alpha^\dagger - \alpha)}$, there is a projection onto $\M^\alpha$ from a forcing 
  of the form $\Add(\mu, \delta - \alpha^\dagger) \times \Q_\alpha$, where $\Q_\alpha$ is $\mu^+$-closed.
\ece
\end{remark}

In \cite{weiss}, Wei\ss\ presents a proof of the fact that, if $\kappa$ is supercompact 
and $\M = \M(\omega, \kappa)$, then, in $V^{\M}$, $\ISP_{\omega_2}$ holds. 
In \cite{holy_lucke_njegomir}, Holy, L\"{u}cke, and
Njegomir point out a mistake in Wei\ss's proof but give an alternate proof of the theorem.
In \cite{viale_weiss}, Viale and Wei\ss\ prove the following theorem.\footnote{The theorem 
is not stated in this form in \cite{viale_weiss}, but it follows immediately from the two 
cited results.}

\begin{theorem}[Viale--Wei\ss\ {\cite[Proposition 6.8 and Theorem 6.9]{viale_weiss}}]
  \label{consistency_strength_thm}
  Suppose that $V \subseteq W$ are transitive models of $\ZFC$ such that
  \begin{itemize}
    \item $\kappa$ is an inaccessible cardinal in $V$;
    \item $(V,W)$ satisfies the $\kappa$-covering and $\kappa$-approximation properties;
    \item for every $\gamma < \kappa$ and $S \subseteq (S^\gamma_\omega)^V$ in $V$, 
    if $S$ is stationary in $V$, then $S$ remains stationary in $W$;
    \item in $W$, $\ITP(\kappa, \geq \kappa)$ holds. 
  \end{itemize}
  Then $\kappa$ is supercompact in $V$.
\end{theorem}

Since for inaccessible $\kappa$ the Mitchell forcing $\M(\omega, \kappa)$ is $\kappa$-Knaster, it follows that 
it satisfies the $\kappa$-covering and $\kappa$-approximation properties. Moreover, 
since there is a projection onto $\M(\omega,\kappa)$ from a forcing of the form 
$\Add(\omega, \kappa) \times \bb{Q}$, where $\bb{Q}$ is countably closed, it follows 
that $\M(\omega, \kappa)$ preserves the stationarity of all stationary subsets 
of $S^\gamma_\omega$ for every ordinal $\gamma$ of uncountable cofinality. Therefore, 
$V \subseteq V^{\M(\omega, \kappa)}$ satisfy the hypothesis of Theorem 
\ref{consistency_strength_thm}. In particular, if $\kappa$ is a strongly compact but 
not supercompact cardinal, then in $V^{\M(\omega, \kappa)}$, there is $\lambda \geq 
\omega_2$ for which $\ITP(\omega_2, \lambda)$ fails.

In light of this, and as part of a broader set-theoretic project to understand the relationship 
between strongly compact and supercompact cardinals, it is natural to ask precisely which 
strong tree properties necessarily 
hold in the strongly compact Mitchell model.
In \cite{weiss}, Wei\ss\ claims without providing a proof that if $\kappa$ is strongly 
compact, then, in $V^{\M(\omega, \kappa)}$, $\SP(\omega_1, \omega_2, \geq \omega_2)$ holds. 
We have been unable to verify the truth of this 
claim; nontrivial problems arise when attempting to use standard techniques to prove it. 
The next two results record the strongest results in this direction we were able to obtain. 
The first is a result about generalized tree properties, indicating that $\TP_{\omega_2} 
(\Lambda)$ holds for every $\omega_2$-directed partial order $\Lambda$. The second 
indicates that a strong version of $\wSP(\omega_1, \omega_2, \geq \omega_2)$ holds, and yields 
Theorem B from the Introduction. 
We present our results in more generality, though, 
in terms of $\M(\mu, \kappa)$ for arbitrary $\mu$ (recall our assumption for this section 
that $\mu^{<\mu} = \mu$).

\begin{theorem}\label{Th:Mitchell_TP}
  Suppose that $\kappa$ is a strongly compact cardinal and $\M = \M(\mu, \kappa)$. 
  Then, in $V^\M$, $\TP_{\mu^{++}}(\Lambda)$ 
  holds for every $\mu^{++}$-directed partial order $\Lambda$.
\end{theorem}

\begin{proof}
  Recall that $\Vdash_{\M}$``$\kappa = \mu^{++}$''. Let $\dot{\Lambda}$ be an $\M$-name for a 
  $\kappa$-directed partial order, and let $\dot{T}$ be an $\M$-name for a 
  $\kappa$-$\dot{\Lambda}$-tree. Also fix an arbitrary condition $(p^0, p^1) \in \M$. Without loss of 
  generality, we can assume that the underlying set of $\dot{\Lambda}$ is forced to be a cardinal, 
  and, by strengthening 
  $(p^0, p^1)$ if necessary, we can assume that there is a cardinal $\lambda$ that is forced by 
  $(p^0, p^1)$ to be the underlying set of $\dot{\Lambda}$.
  
  Let $j:V \rightarrow M$ be an elementary embedding witnessing that $\kappa$ is 
  $\lambda$-strongly compact. In particular, $\mathrm{crit}(j) = \kappa$, $j(\kappa) > \lambda$, 
  and there is a set $x \in (\power_{j(\kappa)}(j(\lambda)))^M$ such that 
  $j''\lambda \subseteq x$.
  Note that $j(\M) = \M(\mu, j(\kappa))^M$ (defined using the set $j(A)$) since $j$ is identity below $\kappa$. Additionally, $\M(\mu, j(\kappa))^M=\M(\mu,j(\kappa))$ since $M$ is closed under sequences of length $\kappa$ and the conditions in $\M(\mu,j(\kappa))$ are functions with domain in $M$ of size $<\kappa$ ranging over $V_\kappa\sub M$.
Moreover, $j \restriction \M$ is the identity function, since $j$ is identity below $\kappa$ and the conditions in $\M$ are bounded in $\kappa$, and, by clause (\ref{quotient_approx_prop}) 
  of Remark \ref{mitchell_remark} applied to $\M(\mu, j(\kappa))$, there is a 
  projection from $j(\M)$ to $\M$ such that the quotient forcing is forced to have the 
  $\mu^+$-approximation property. Let $G$ be $\M$-generic over $V$, and let $g$ be 
  $j(\M)/\M$ generic over $V[G]$. Then, in $V[G][g]$, we can lift $j$ to 
  $j:V[G] \rightarrow M[G][g]$.
  
  In $V[G]$, let $\Lambda$ and $T$ be the realizations of $\dot{\Lambda}$ and $\dot{T}$, respectively.
  In $M[G][g]$, $j(\Lambda)$ is a $j(\kappa)$-directed partial order and $j(T)$ is a 
  $j(\kappa)$-$j(\Lambda)$-tree. Moreover, for all $u \in \Lambda$, since 
  $\vert T_u\vert  < \kappa$ and $\mathrm{crit}(j) = \kappa$, we have $j(T)_{j(u)} = j''T_u$.
  Since $\vert x\vert ^{M[G][g]} < j(\kappa)$, we can therefore find a $w \in j(\Lambda)$ such that 
  $v <_{j(\Lambda)} w$ for all $v \in x$. Fix an arbitrary $s \in j(T)_w$, and define a function 
  $b \in \prod_{u \in \Lambda} T_u$ by letting $b(u)$ be the unique $t \in T_u$ such that 
  $j(t) = s \restriction j(u)$. 
  
  Note that, if $u <_\Lambda u'$, then 
  $j(b(u)), j(b(u')) <_{j(T)} s$, so $j(b(u)) <_{j(T)} j(b(u'))$ and hence, by elementarity, 
  $b(u) <_T b(u')$. It follows that $b$ is a cofinal branch through $T$. Moreover, $b \in V[G][g]$, 
  but $g$ is generic over $V[G]$ for $j(\M)/\M$, which has the $\mu^+$-approximation property and 
  hence, \emph{a fortiori}, the $\kappa$-approximation property. Therefore, by Lemma 
  \ref{approximation_branch_lemma}¸ we have $b \in V[G]$, and hence $\TP_\kappa(\Lambda)$ holds 
  in $V[G]$.
\end{proof}

\begin{theorem} \label{strongly_compact_thm_2}
  Suppose that $\kappa$ is a strongly compact cardinal, $\lambda \geq \kappa$ 
  is regular, and $\Sigma \subseteq [\kappa, \lambda]$ is a set of regular cardinals 
  such that $\vert \Sigma\vert  < \kappa$. Let $\M = \M(\mu, \kappa)$. Then, in $V^\M$, 
  $\wSP_{\mc Y}(\mu^+, \kappa, \lambda)$ holds, where $\mc Y$ denotes the set of 
  $x \in \power_\kappa \lambda$ such that $\cf(\sup(x \cap \chi)) = \mu^+$ for all 
  $\chi \in \Sigma$.
\end{theorem}

\begin{proof}
  By enlarging $\Sigma$ if necessary, we can assume that $\kappa \in \Sigma$.
  Let $\dot{\mc Y}$ be an $\M$-name for $\mc Y$, and let 
  $\dot{D} = \langle \dot{d}_{\dot{x}} \mid 
  \dot{x} \in \power_\kappa \lambda \rangle$ be an $\M$-name for a strongly 
  $\mu^+$-$\dot{\mc Y}$-slender $(\kappa, \lambda)$-list.\footnote{This is 
  a slight abuse of notation for the sake of readability. Formally, $\dot{D}$ is an 
  $\M$-name for a strongly $\mu^+$-$\dot{\mc Y}$-slender $(\kappa, \lambda)$-list 
  and, for every nice $\M$-name $\dot{x}$ for an element of $(\power_\kappa 
  \lambda)^{V^{\M}}$, $\dot{d}_{\dot{x}}$ is an $\M$-name for the 
  $\dot{x}$-th entry in $\dot{D}$. Similar notation will be used throughout 
  the remainder of the paper.} Let $G$ be $\M$-generic over $V$ 
  and, for each inaccessible $\alpha < \kappa$, let $G_\alpha$ be the $\M_\alpha$-generic 
  filter induced by $G$. Let $\mc Y$ and $D$ be the realizations of $\dot{\mc Y}$ and 
  $\dot{D}$, respectively, in $V[G]$.
  
  The first part of our proof largely follows the beginning of the proof of \cite[Theorem 5.4]{weiss}. 
  Work for now in $V[G]$. By the $\kappa$-c.c.\ of $\M$, we know that, for all 
  $x \in \power_\kappa \lambda$, there is an inaccessible $\alpha_x < \kappa$ such that 
  $d_x \in V[G_{\alpha_x}]$. Let $\theta > \lambda$ be a sufficiently large regular cardinal, 
  and let $C \subseteq \power_\kappa H(\theta)$ be a strong club witnessing that 
  $D$ is strongly $\mu^+$-$\mc Y$-slender. We can assume that, for all $N \in C$, we have 
  $\kappa_N := \kappa \cap N \in \kappa$ and, for all $x \in \power_\kappa \lambda \cap N$, 
  we have $\alpha_x \in N$. It follows that, 
  whenever $N \in C$, $\kappa_N$ is inaccessible in $V$, and 
  $x \in \power_\kappa \lambda \cap N$, we have $x \in V[G_{\kappa_N}]$.

  In $V$, let $\sigma := \lambda^{<\kappa}$ and find $\bar{N} \prec H(\theta)$ such that 
  $\power_\kappa \lambda \subseteq \bar{N}$ and $\vert \bar{N}\vert  = \sigma$. In $V[G]$, let 
  $C \restriction \bar{N} := \{N \cap \bar{N} \mid N \in C\}$. Then $C \restriction 
  \bar{N}$ is a strong club in $\power_\kappa \bar{N}$. Now move back to $V$, and 
  let $\dot{C}$ be an $\M$-name for $C$. Since $\M$ has the $\kappa$-c.c., we can find 
  a strong club $E \subseteq \power_\kappa \bar{N}$ such that $\Vdash_{\M}$``$E \subseteq 
  \dot{C} \restriction \bar{N}$''. 
  
  Let $j:V \rightarrow M$ be an elementary embedding witnessing that $\kappa$ is 
  $\sigma$-strongly compact. Therefore, in $M$, we can find a set $y \in 
  \power_{j(\kappa)}j(\bar{N})$ such that $j''\bar{N} \subseteq y$. 
  For each $\chi \in \Sigma$, let $\gamma_\chi := \sup j''\chi$. Note that, since 
  $\vert \Sigma\vert  < \kappa$, we know that both $\{\gamma_\chi \mid \chi \in 
  \Sigma\}$ and $j \restriction \Sigma$ are in $M$. Let 
  \[
  Z := \{z \subseteq y \mid z \in j(E) \text{ and, for all } \chi \in \Sigma, ~ 
  z\cap j(\chi) \subseteq \gamma_\chi\}.
  \] 
  Then $Z \in M$ and, in $M$, it is a 
  subset of $j(E)$ of cardinality less than $j(\kappa)$. Therefore, 
  $\bigcup Z \in j(E)$. Let $w := j(\lambda) \cap \bigcup Z$.
  The definition of $Z$ implies that 
  $j''E \subseteq Z$ and therefore $j''\lambda \subseteq w$. It follows that, 
  for all $\chi \in \Sigma$, we have $\sup(w \cap j(\chi)) = \gamma_\chi$.
  
  Note that $j(\M) = \M(\mu, j(\kappa))$, defined using the set $j(A)$. Let $g$ be $j(\M)/\M$-generic over 
  $V[G]$, and lift $j$ to $j:V[G] \rightarrow M[G][g]$. In $M[G][g]$, for every 
  $V$-regular cardinal $\chi \in [\kappa, j(\kappa))$, we have $\cf(\chi) = \mu^+$. 
  In particular, for every $\chi \in \Sigma$, we have $\cf(\gamma_\chi) = \cf(j''\chi) = 
  \cf(\chi) = \mu^+$, and hence $w \in j(\mc Y)$.
  
  Let $j(D) = \langle d'_x \mid x \in \power_{j(\kappa)} j(\lambda) \rangle$. 
  Since $\bigcup Z \in j(E)$, we can find $N \in j(C)$ such that 
  $N \cap j(\bar{N}) = \bigcup Z$. In particular, $N \cap j(\lambda) = w$. 
  Therefore, since $j(D)$ is a $\mu^+$-$j(\mc Y)$-slender $(j(\kappa), j(\lambda))$-list, 
  as witnessed by the strong club $j(C)$, it follows that, for every  
  $u \in N \cap \power_{\mu^+} j(\lambda)$, we have $d'_w \cap u 
  \in N$. Note also that $j(\kappa) \cap N = \kappa$ is inaccessible in $V$, and hence 
  we have $d'_w \cap u \in V[j(G)_\kappa] = V[G]$ for all $u \in N \cap 
  \power_{\mu^+} j(\lambda)$.
  
  \begin{claim} \label{u_claim}
    For all $u \in (\power_{\mu^+} \lambda)^V$, we have $j(u) \in N$.
  \end{claim}  
  
  \begin{proof}
    Fix $u \in (\power_{\mu^+} \lambda)^V$. Since $E$ is cofinal in $\power_\kappa \bar{N}$ 
    and $(\power_{\mu^+} \lambda)^V \subseteq \bar{N}$, we can find 
    $z \in E$ such that $u \in z$. Then $j(z) \in Z$, and hence $j(u) \in 
    \bigcup Z \subseteq N$.
  \end{proof}
  
  Let $b := \{\alpha < \lambda \mid j(\alpha) \in d'_w\}$.
  
  \begin{claim} \label{u_approximation_claim}
    For all $u \in (\power_{\mu^+} \lambda)^{V[G]}$, we have $b \cap u \in V[G]$.
  \end{claim}
  
  \begin{proof}
    Fix $u \in (\power_{\mu^+} \lambda)^{V[G]}$. Since $\M$ satisfies the $\mu^+$-covering 
    property, we can find $u' \in (\power_{\mu^+} \lambda)^V$ such that $u \subseteq u'$.
    By Claim \ref{u_claim}, we have $j(u') \in N$. Then, by the sentence preceding 
    Claim \ref{u_claim}, we have $d'_w \cap j(u') \in V[G]$. But then 
    $b \cap u = \{\alpha \in u \mid j(\alpha) \in d'_w \cap j(u')\}$ is definable in 
    $V[G]$, and hence is in $V[G]$.
  \end{proof}
  
  In $V[G]$, the quotient forcing $j(\M)/\M$ has the $\mu^+$-approximation property. 
  Therefore, Claim \ref{u_approximation_claim} implies that $b \in V[G]$. We will therefore 
  be done if we can show that, in $V[G]$, $b$ is a cofinal branch through $D$. To this end, 
  fix $x \in (\power_\kappa \lambda)^{V[G]}$, and note that $j(x) = j''x$ and 
  $j(b \cap x) = j''(b \cap x)$. Now 
  \[
    M[G][g] \models \exists z \in \power_{j(\kappa)} j(\lambda) ~ [j(x) \subseteq z 
    \text{ and } d'_z \cap j(x) = j(b \cap x)],
  \]
  as witnessed by $w$. Therefore, by elementarity of $j$, we have
  \[
    V[G] \models \exists z \in \power_\kappa \lambda ~ [x \subseteq z \text{ and } 
    d_z \cap x = b \cap x].
  \]
  Since $x$ was arbitrary, this implies that $b$ is in fact a cofinal branch of $D$, 
  thus completing the proof.
\end{proof}

\begin{remark} \label{wsp_remark}
  Suppose that $\theta \geq \kappa$ is a regular uncountable cardinal and $f:\vert H(\theta)\vert  
  \rightarrow H(\theta)$ is a bijection. Then 
  \[
    C := \{x \in \power_\kappa \vert H(\theta)\vert  \mid f[x] \cap \theta = x \cap \theta\}
  \]
  is a strong club in $\power_\kappa \vert H(\theta)\vert $ (and hence $\{f[X] \mid X \in C\}$ is a strong 
  club in $\power_\kappa H(\theta)$). Therefore, by the discussion following Remark 
  \ref{tp_convention_remark}, we see that, in the forcing extension of the previous theorem, 
  for every regular $\theta \geq \kappa$ and every collection $\Sigma \subseteq [\kappa, \theta]$ of 
  regular cardinals with $\vert \Sigma\vert  < \kappa$, we have $\wSP_{\mc Y}(\mu^+, \kappa, H(\theta))$, 
  where $\mc Y$ is the set of $M \in \power_\kappa H(\theta)$ such that 
  $\cf(\sup(M \cap \chi)) = \mu^+$ for all $\chi \in \Sigma$. In particular, 
  by the analogue of Theorem \ref{slender_guessing_thm_1} for $\wSP$ and 
  $\wAGP$, Theorem \ref{subadditive_function_theorem}, and Corollary \ref{square_failure_cor},
  the version of $\wSP$ obtained in the preceding theorem is enough to obtain the 
  nonexistence of subadditive, strongly unbounded functions $c:[\lambda]^2 \rightarrow \chi$ 
  for every regular $\lambda \geq \kappa$ and every $\chi$ with $\chi^+ < \kappa$, and 
  hence the failure of $\square(\lambda)$ for every regular $\lambda \geq \kappa$. 
\end{remark}

We now present a result indicating that, if $\kappa$ is strongly compact, then, in the 
extension by $\M(\omega, \kappa)$, an instance of $\wSP(\ldots)$ is indestructible under 
certain iterations of small c.c.c.\ forcings. This is a variation on a result of Todor\v{c}evi\'{c} 
from \cite{T:wKH}.

\begin{theorem} \label{ccc_iteration_thm}
  Suppose that $\kappa$ is strongly compact, $\M = \M(\omega, \kappa)$, and, in $V^{\M}$, 
  $\R := \langle \R_\alpha, 
  \dot{\bb{S}}_\beta \mid \alpha \leq \kappa, ~ \beta < \kappa \rangle$ is a finite-support 
  forcing iteration such that, for all $\beta < \kappa$, $\dot{\bb{S}}_\beta$ is 
  forced to be a c.c.c.\ forcing of size at most $\aleph_1$ that does not add cofinal 
  branches to any Suslin tree. Suppose also that $\lambda \geq \kappa$ is regular 
  and $\Sigma \subseteq [\kappa, \lambda]$ is a set of regular cardinals such that 
  $\vert \Sigma\vert  < \kappa$. Then, in $V^{\M \ast \dot{\R}}$, $\wSP_{\mc Y}(\omega_1, \omega_2, 
  \lambda)$ holds, where $\mc Y$ denotes the set of $x \in \power_{\omega_2} \lambda$ 
  such that $\cf(\sup(x \cap \chi)) = \omega_1$ for all $\chi \in \Sigma$.
\end{theorem}

\begin{proof}
  The proof repeats many of the steps of the proof of Theorem \ref{strongly_compact_thm_2}, 
  so we will leave out some details. We may again assume that $\kappa$ is in $\Sigma$. Let 
  $\dot{\mc{Y}}$ be an $\M \ast \dot{\R}$-name for $\mc Y$, and let $\dot{D} = 
  \langle d_{\dot{x}} \mid \dot{x} \in \power_\kappa \lambda \rangle$ be an 
  $\M \ast \dot{\R}$-name for a strongly $\omega_1$-$\dot{\mc Y}$-slender 
  $(\omega_2, \lambda)$-list. Let $G*H$ be $\M \ast \dot{\R}$-generic over $V$ and, 
  for each inaccessible $\alpha < \kappa$, let $G_\alpha$ and $H_\alpha$ be the 
  $\M_\alpha$ and $\R_\alpha$-generic filters induced by $G$ and $H$, respectively.
  We can assume that the underlying set of each $\dot{\bb{S}}_\beta$ is forced to 
  be $\omega_1$ and $\dot{\bb{S}}_\beta$ is a nice name for a subset of $\omega_1^2$. 
  By the $\kappa$-c.c.\ of $\M$, there is a club $B \subseteq \kappa$ 
  such that, for every $\alpha \in B$, $\dot{\R}_\alpha$ is an $\M_\alpha$-name.
  Therefore, for each $\alpha \in B$, it makes sense to speak about the forcing extension 
  $V[G_\alpha][H_\alpha]$.
  
  Let $\theta > \lambda$ be a sufficiently large regular cardinal. As in the proof 
  of Theorem \ref{strongly_compact_thm_2}, working in $V[G][H]$, we can 
  find a strong club $C \subseteq \power_\kappa H(\theta)$ such that, for every 
  $N \in C$, we have $\kappa_N := \kappa \cap N \in \kappa$ and, whenever 
  $N \in C$, $\kappa_N$ is inaccessible in $V$, and $x \in \power_\kappa \lambda \cap N$, 
  we have $x \in V[G_{\kappa_N}][H_{\kappa_N}]$.

  Work now in $V$, and let $\sigma$, $\bar{N}$, $C \restriction \bar{N}$, $E$, 
  $j:V \rightarrow M$, and $w$ be defined as in the proof of Theorem 
  \ref{strongly_compact_thm_2}. Note that $j(\M) = \M(\omega, j(\kappa))$ and as in the proof of Theorem \ref{Th:Mitchell_TP}, $j(\M)$ is forcing equivalent to $\M * j(\M)/\M$. The forcing $j(\dot{\R})$ is of the form $\dot{\R} \ast \dot{\bb{U}}$, where $\dot{\bb{U}}$ is 
  forced to be a finite-support iteration of length $j(\kappa)$ such that each 
  iterand is forced to be a c.c.c.\ forcing of size at most $\aleph_1$ that does not 
  add cofinal branches to any Suslin tree. Note that $j(\M\ast\dot{\R)}=j(\M)\ast j(\dot{\R})$, and therefore $j(\M*\dot{\R})$ is forcing eqivalent to $\M \ast j(\M)/\M\ast\dot{\R}\ast\dot{\bb{U}}$. Both $j(\M)/\M$ and $\R$ are defined in $V[G]$ and therefore $\M \ast j(\M)/\M\ast\dot{\R}\ast\dot{\bb{U}}$ is forcing equivalent to $\M \ast \dot{\R}\ast j(\M)/\M\ast\dot{\bb{U}}$. Let $g * h$ be $ j(\M)/\M\ast\dot{\bb{U}}$-generic over $V[G][H]$, and lift $j$ to $j:V[G][H] \rightarrow M[G][H][g][h]$.
  In $V[G][H][g][h]$, let $j(D) = \langle d'_x \mid x \in \power_{j(\kappa)} j(\lambda) 
  \rangle$, and define $b := \{\alpha < \lambda \mid j(\alpha) \in d'_w\}$. 
  
  As in the proof of Theorem \ref{strongly_compact_thm_2}, one can show that, 
  for all $u \in (\power_{\omega_1} \lambda)^{V[G][H]}$, we have $b \cap u \in V[G][H]$ 
  (this uses the fact that $\M \ast \dot{\R}$ has the $\omega_1$-covering property, 
  which follows from the fact that $\M$ has the $\omega_1$-covering property and 
  $\dot{\R}$ is forced to have the c.c.c.).
  
  \begin{claim} \label{approx_claim}
    For all $v \in (\power_\kappa \lambda)^{V[G][H]}$, we have $b \cap v \in V[G][H]$.
  \end{claim}
  
  \begin{proof}
    Since $\R$ has the $\kappa$-covering property in $V[G]$, it suffices to prove the 
    claim for all $v \in (\power_\kappa \lambda)^{V[G]}$. Let $\kappa^{\dagger}$ be the least 
    inaccessible cardinal greater than $\kappa$. In $V[G]$, by Clause \ref{quotient_product_prop} 
    of Remark \ref{mitchell_remark} applied to $j(\M)$, the quotient forcing 
    $j(\M)/\M$ is of the form $\Add(\omega, \kappa^{\dagger} - \kappa) \ast \dot{\M}^\kappa$. Let 
    $g_0$ be the $\Add(\omega, \kappa^{\dagger} - \kappa)$-generic filter induced by $g$. In $V[G][g_0]$, 
    there is a projection onto $\M^\kappa$ from a forcing of the form $\Add(\omega, j(\kappa) - \kappa^\dagger) 
    \times \Q_\kappa$, where $\Q_\kappa$ is $\omega_1$-closed. Let $g_1 \times g_2$ be 
    $(\Add(\omega, j(\kappa) - \kappa^\dagger) \times \Q_\kappa)$-generic over $V[G][g_0]$ such that $V[G][g_0][g_1][g_2]$ is an 
    extension of $V[G][g]$ by $(\Add(\omega, j(\kappa) - \kappa^\dagger) \times \Q_\kappa)/\M^\kappa$.

    Fix $v \in (\power_\kappa \lambda)^{V[G]}$, and let $\Lambda := 
    (\power_{\omega_1} v)^{V[G]}$. Since all forcing notions used in this proof have the 
    $\omega_1$-covering property, $\Lambda$ is $\omega_1$-directed in all of the forcing 
    extensions of $V[G]$ appearing here. Now consider the $\Lambda$-tree $T := \langle \langle 
    T_u \mid u \in \Lambda \rangle, \subseteq \rangle$, where $T_u := (\power(u))^{V[G][H]}$ 
    for every $u \in \Lambda$. Note that $\vert T_u\vert  = (2^{\omega_1})^{V[G][H]} = \kappa < \kappa^{\dagger}$.
    
    By the paragraph immediately preceding the claim, $b \cap v$ is a 
    branch through $T$, and we know that $b \cap v \in V[G][H][g][h]$.\footnote{More precisely, 
    the function $u \mapsto b \cap u$ defined on $\Lambda$ is a branch through $T$, but it is clear 
    that this function and $b$ are definable from one another in all models of interest.} 
    First note that, in $V[G][H][g]$, $j(\R)/H$ is equivalent to a tail of the iteration 
    $j(\R)$, which is a finite-support iteration of c.c.c.\ posets that do not add branches 
    to any Suslin tree. By \cite[Lemma 3.7]{D:trees}, it follows that the entire iteration $j(\R)/H$ 
    cannot add branches to any Suslin tree. Therefore, by Lemma 
    \ref{suslin_preservation_lemma}, $b \cap v$ could not have been added by $h$, so we have 
    $b \cap v \in V[G][H][g]$, and, \emph{a fortiori}, $b \cap v \in V[G][H][g_0][g_1][g_2]$.
   
    We now work in the model $V[G][g_0]$. Since $g_0$ is generic for an $\omega_1$-Knaster poset, 
    we know that $\R$ remains c.c.c.\ in $V[G][g_0]$. Also, $\Q_\kappa$ is $\omega_1$-closed. In $V[G][g_0]$, $2^\omega = \kappa^\dagger > \kappa$, 
    and, in $V[G][g_0][H][g_1]$, $T$ is a $\Lambda$-tree of width $\kappa^+$. We can therefore apply 
    Lemma \ref{closed_preservation_lemma} to conclude that $b \cap v$ cannot have been added by 
    $g_2$ and therefore lies in $V[G][g_0][H][g_1]$. Finally, $g_0 \ast g_1$ is generic over 
    $V[G][H]$ for $\Add(\omega, \kappa^\dagger - \kappa) \ast \Add(\omega, j(\kappa) - \kappa^\dagger) \cong \Add(\omega, 
    j(\kappa)-\kappa)$, so, again by Lemma \ref{suslin_preservation_lemma}, $b \cap v$ cannot have been 
    added by $g_0 \ast g_1$ and therefore lies in $V[G][H]$, as desired.
  \end{proof}  
  
  \begin{claim}
    $b$ is a cofinal branch through $D$.
  \end{claim}
  
  \begin{proof}
    Fix $v \in (\power_\kappa \lambda)^{V[G][H]}$. We must find $x \in 
    (\power_\kappa \lambda)^{V[G][H]}$ such that $v \subseteq x$ and $d_x \cap v = b \cap v$. 
    By Claim \ref{approx_claim}, we have $b \cap v \in V[G][H]$. Moreover, by the definition of 
    $b$, we have 
    \[
      j(b \cap v) = j''(b \cap v) = d'_w \cap j(v).
    \]
    Therefore, 
    \[
      M[G][H][g][h] \models \exists x \in \power_{j(\kappa)} j(\lambda) ~ [d'_x \cap j(v) = j(b \cap v)],
    \]
    as witnessed by $x = w$. Therefore, by elementarity, 
    \[
      V[G][H] \models \exists x \in \power_\kappa \lambda ~ [d_x \cap v = b \cap v],
    \]
    as desired.
  \end{proof}
  
  Now the verification that $b$ is in $V[G][H]$ follows reasoning as in the proof of Claim 
  \ref{approx_claim}. Let $\Lambda'$ be the partial order $((\power_\kappa \lambda)^{V[G][H]}, \subseteq)$, 
  and let $T'$ be the $\Lambda'$-tree generated from $D$ as in Remark \ref{list_to_tree_remark}. 
  Then $\Lambda'$ is $\kappa$-directed in all extensions of $V[G][H]$ appearing in this proof, 
  and, in $V[G][H]$, $T'$ is a $\Lambda'$-tree of width $((\kappa^{<\kappa})^{+})^{V[G][H]} = (\kappa^+)^{V[G][H]} < \kappa^\dagger$. 
  
  Just as in the proof of Claim \ref{approx_claim}, Lemma \ref{suslin_preservation_lemma} implies that 
  $b$ cannot have been added by $h$, so we have $b \in V[G][H][g]$, and hence also 
  $b \in V[G][H][g_0][g_1][g_2]$. An application of Lemma \ref{closed_preservation_lemma} in 
  $V[G][g_0]$ shows that $b$ must be in $V[G][H][g_0][g_1]$, and the fact that $g_0 \ast g_1$ is 
  generic for $\Add(\omega, j(\kappa) - \kappa)$ over $V[G][H]$ implies that $b$ must be in $V[G][H]$. 
  Therefore, $D$ has a cofinal branch in $V[G][H]$, so we have verified this instance of 
  $\wSP_{\mc Y}(\omega_1, \omega_2, \lambda)$ in $V[G][H]$.
\end{proof}

As a corollary, we show that Martin's Axiom together with strong tree properties at $\omega_2$ can 
be obtained starting only from a strongly compact cardinal (both are consequences of $\PFA$, which 
can be forced from a supercompact cardinal).

\begin{corollary}
  Suppose that $\kappa$ is a strongly compact cardinal. Then there is a forcing extension in which 
  $\kappa = \omega_2 = 2^\omega$ and $\MA + \wSP(\omega_1, \omega_2, \geq \omega_2)$ (and hence 
  $\MA + \TP(\omega_2, \geq \omega_2)$) holds.
\end{corollary}

\begin{proof}
  In $V$, let $\M := \M(\omega, \kappa)$. Now work in $V^{\M}$. In \cite{D:trees}, Devlin 
  constructed a finite-support iteration $\R$ of ccc forcings of size at most $\aleph_1$ that 
  do not add cofinal branches to any Suslin tree and such that $\Vdash_{\R} \MA$. Intuitively speaking, 
  $\R$ is the standard iteration to force $\MA$, except that, when one encounters a ccc forcing of size 
  $\aleph_1$ that adds a branch to a Suslin tree, one instead forces to specialize the Suslin tree; we 
  refer the reader to \cite[Section 3]{D:trees} (and also \cite{T:wKH}) for details. Then $\MA$ holds 
  in $V^{\M \ast \dot{\R}}$. By Theorem \ref{ccc_iteration_thm}, $\wSP(\omega_1, \omega_2, \geq 
  \omega_2)$, and hence $\TP(\omega_2, \geq \omega_2)$, holds in $V^{\M \ast \dot{\R}}$ as well.
\end{proof}

\begin{remark}
  As in Remark \ref{wsp_remark}, the strengthening of $\wSP(\omega_1, \omega_2, 
  \geq \omega_2)$ obtained in Theorem \ref{ccc_iteration_thm} 
  is enough to yield, for instance, the failure of 
  $\square(\lambda)$ for all regular $\lambda \geq \omega_2$.
\end{remark}

\section{The slender tree property and Kurepa trees} \label{kurepa_section}

In this section, we investigate the influence of various strong tree properties on the existence of 
(weak) Kurepa trees and prove Theorem D, in the process showing that, for instance, $\ISP(\omega_2, \omega_2, \geq \omega_2)$ 
does not imply $\SP(\omega_1, \omega_2, \vert H(\omega_2)\vert )$, i.e., the monotonicity in the first 
coordinate of these principles is in general strict. We first recall the notion of a \emph{(weak) 
$\mu$-Kurepa tree}.

\begin{definition}
  Let $\mu$ be a regular uncountable cardinal.
  \bce[(i)]
    \item A \emph{$\mu$-Kurepa tree} is a $\mu$-tree (i.e., a tree of height $\mu$, all of whose 
    levels have size less than $\mu$) with at least $\mu^+$-many cofinal branches.
    \item A \emph{weak $\mu$-Kurepa tree} is a tree of height and cardinality $\mu$ with at 
    least $\mu^+$-many cofinal branches. 
  \ece
\end{definition}

The following proposition is easily verified, so we leave its proof to the reader.

\begin{proposition} \label{cofinal_strong_club_prop}
  Let $\kappa \leq \theta$ be regular uncountable cardinals, and suppose that $S \subseteq 
  \power_\kappa H(\theta)$ is cofinal. Then the set of elements of $\power_\kappa H(\theta)$ that 
  can be written as unions of elements of $S$ is a strong club in $\power_\kappa H(\theta)$.
\end{proposition}

The following theorem is clause (1) of Theorem C.

\begin{theorem}\label{th:SP}
  Suppose that $\mu$ is a regular uncountable cardinal. If the principle $\wAGP(\mu, \mu^+, \mu^+)$ holds, 
  then there are no weak $\mu$-Kurepa trees.
\end{theorem}

\begin{proof}
  Let $T$ be a tree of height and size $\mu$. We will show that $T$ has at most $\mu$-many cofinal 
  branches. Without loss of generality, assume that the underlying set of $T$ is 
  $\mu$. Let $S$ be the set of $N \in \power_{\mu^+} H(\mu^+)$ such that $N \prec H(\mu^+)$, 
  $T \in N$, and $\mu \subseteq N$. By $\wAGP(\mu, \mu^+, \mu^+)$, and recalling 
  Proposition \ref{cofinal_strong_club_prop}, we can find $M \in \power_{\mu^+} 
  H(\mu^+)$ such that $T \in M$, $\mu \subseteq M$, $M$ can be written as a union of elements of 
  $S$, and $(M, \mu)$ is almost $\mu$-guessed by $S$.
  
  \begin{claim} \label{branch_approximation_claim}
    Let $b \subseteq \mu$ be a cofinal branch in $T$. Then $b$ is $(\mu, M)$-approximated.
  \end{claim}
  
  \begin{proof}
    Fix a set $z \in M \cap \power_\mu \mu$, and consider the set $b \cap z$. 
    Since $M$ can be written as the union of elements of $S$, we can find $N \in S$ such that 
    $N \subseteq M$ and $z \in N$. Since $\vert z\vert  < \mu$, we can find $\gamma \in b$ such that 
    $\alpha <_T \gamma$ for all $\alpha \in b \cap z$. Then $b \cap z$ is definable in $N$ as 
    $\{\alpha \in z \mid \alpha <_T \gamma\}$. Therefore, $b \cap z \in N$, and, since 
    $N \subseteq M$, we have $b \cap z \in M$, as well. Therefore, $b$ is $(\mu, M)$-approximated.
  \end{proof}
  
  Fix a cofinal branch $b$ in $T$. By Claim \ref{branch_approximation_claim}, $b$ is 
  $(\mu, M)$-approximated. Therefore, since $(M, \mu)$ is almost $\mu$-guessed by $S$, we can 
  find $N \in S$ such that $N \subseteq M$ and $b$ is $N$-guessed, i.e., there is $d \in N$ 
  such that $d \cap N = b \cap N$. Since $b \subseteq \mu$, we can assume that $d \subseteq \mu$. 
  But then, since $\mu \subseteq N$, we must have $d = b$. Therefore, $b \in N$ and, since 
  $N \subseteq M$, we also have $b \in M$. Therefore, every cofinal branch of $T$ is in $M$, so, 
  since $\vert M\vert  \leq \mu$, there are at most $\mu$-many cofinal branches of $T$.
\end{proof}

We record the following immediate corollary to provide some additional context for Theorem 
\ref{th:SPconverse}.

\begin{corollary}
  If $\SP(\omega_1, \omega_2, \vert H(\omega_2)\vert )$ holds, then there is no weak $\omega_1$-Kurepa tree.
\end{corollary}

\begin{proof}
  This is immediate from the definitions, from Theorem \ref{slender_guessing_thm_1}, and from 
  Theorem \ref{th:SP}.
\end{proof}

In preparation for the proof of Theorem \ref{th:SPconverse}, we review the forcing for adding an $\omega_1$-Kurepa tree, which we denote $\K(\omega_1,\delta)$. We use the definition from \cite{C:trees}.

\begin{definition}
Let $\delta \geq \omega_2$ be an ordinal.
A condition $p$ in $\K(\omega_1,\delta)$ is a pair $(t_p,f_p)$ where the following hold:
\bce[(i)]
\item $t_p$ is a countable normal tree of successor height $\alpha_p+1$;
\item $f_p$ is a countable partial function from $\delta$ to $(t_p)_{\alpha_p}$.
\ece

 A condition $q$ is stronger than $p$ if the following hold:
\bce[(i)]
\item $\alpha_q\ge \alpha_p$ and $t_q\rest (\alpha_p +1)=t_p$;
\item $\dom {f_p}\sub\dom {f_q}$ and $f_p(\xi)\le_{t_q} f_q(\xi)$ for all $\xi \in \dom {f_p}$.
\ece
\end{definition}

Proofs of the following facts can be found in \cite{C:trees}:

\begin{fact}
The forcing $\K(\omega_1, \delta)$ is $\omega_1$-closed and it adds an $\omega_1$-Kurepa tree with $\vert \delta \vert$-many branches.
\end{fact}

\begin{fact}
Let $\mu>\omega$ be a regular cardinal such that $\gamma^\omega<\mu$ for all $\gamma<\mu$. Then the forcing $\K(\omega_1, \delta)$ is $\mu$-Knaster.
\end{fact}

If $\delta<\delta'$, we can define a projection $\pi:\K(\omega_1,\delta')\then\K(\omega_1,\delta)$ by $\pi(t,f)=(t,f\rest\delta)$. If $H$ is $\K(\omega_1,\delta)$-generic, $\pi$ determines the quotient forcing $\K(\omega_1,\delta')/H=\set{p\in \K(\omega_1,\delta')}{\pi(p)\in H}$. Note that the tree $T=\bigcup\set{t}{(t,f)\in H}$ is already added by $H$: the quotient $\K(\omega_1,\delta')/H$ just adds new cofinal branches to $T$. Working in $V[H]$, it is easy to verify that the quotient $\K(\omega_1,\delta')/H$ is forcing equivalent to the following forcing which we denote $\K(\omega_1,\delta'\setminus\delta)$. A condition in $\K(\omega_1,\delta'\setminus\delta)$ is a pair $p=(\alpha_p,f_p)$, where $\alpha_p<\omega_1$ and $f_p$ is a countable partial function from $\delta'\setminus\delta$ to $T_{\alpha_p}$. We say that $q$ is stronger than $p$ if $\alpha_p\le \alpha_q$, $\dom{f_p}\sub\dom {f_q}$ and $f_p(\xi)\le_T f_q(\xi)$ for all $\xi\in\dom{ f_p}$.

\begin{theorem}\label{th:SPconverse}
Let $\kappa$ be a supercompact cardinal, and assume that $\GCH$ holds above $\kappa$. Then there is a generic extension where the following hold:
\bce[(i)]
\item $2^\omega=\kappa=\omega_2$;
\item there is an $\omega_1$-Kurepa tree;
\item $\ISP(\omega_2,\omega_2,\geq \omega_2)$ holds.

\ece
\end{theorem}

\begin{proof}
The generic extension is obtained by forcing with a product of the Mitchell forcing and the forcing for adding an $\omega_1$-Kurepa tree: $\M(\omega,\kappa)\times\K(\omega_1,\kappa)$, which we denote by $\M\times\K$ for simplicity. Let $G\times H$ be $\M\times \K$-generic over $V$. The Mitchell forcing preserves $\omega_1$ and all cardinal greater or equal to $\kappa$ and forces $2^\omega=\kappa=\omega_2$. Since $\K$ is $\kappa$-Knaster in $V$, $\K$ is still $\kappa$-Knaster in $V[G]$ and therefore it preserves all cardinals greater or equal to $\kappa$. It also preserves $\omega_1$ over $V[G]$, since it is $\omega_1$-closed in $V$ and therefore it is $\omega_1$-distributive in $V[G]$ by the standard product analysis of the Mitchell forcing $\M$ (for more details see \cite[Lemma 1.10]{HS:tp}). This finishes the proof of (i).

There is an $\omega_1$-Kurepa tree $T$ in $V[H]$ with $\kappa$-many cofinal branches by definition of $\K$. Since $\omega_1$ is preserved by $\K\times\M$, $T$ is still $\omega_1$-tree with $\kappa=\omega_2$-many cofinal branches in $V[H][G]=V[G][H]$.

Now fix $\lambda > \kappa$; by our hypotheses, we can assume that $\lambda^{<\lambda} = \lambda$. To show that  $\ISP(\kappa,\kappa,\lambda)$ holds in $V[G][H]$ we use the lifting argument and analysis of the quotient as in \cite{C:trees}. Let \begin{equation}\label{eq:j} j: V\then M \end{equation} be a supercompact elementary embedding with critical point $\kappa$ given by a normal ultrafilter $U$ on $\power_\kappa(\lambda)$; i.e. $M\iso \Ult(V,U)$. 
We lift the elementary embedding $j$ in two steps. Note that $j(\M(\omega,\kappa))=\M(\omega,j(\kappa))$ and recall that there is a projection from $\M(\omega,j(\kappa))$ to $\M(\omega,\kappa)$. Let us denote the quotient $\M(\omega,j(\kappa))/G$ by $Q_\M$ and let $g$ be $Q_\M$-generic over $V[G][H]$. Then we can lift in $V[G][g]$ the embedding to $j:V[G]\then\M[G][g]$. Now $j(\K(\omega_1,\kappa))=\K(\omega_1,j(\kappa))$ and there is a projection from $\K(\omega_1,j(\kappa))$ to $\K(\omega_1,\kappa)$.  Let us denote the quotient $\K(\omega_1,j(\kappa))/G$ by $Q_\K$ and let $h$ be $Q_\K$-generic over $V[G][H][g]$. We can lift the embedding in $V[G][H][g][h]$ further to \begin{equation}\label{eq:lift} j:V[G][H]\then M[G][H][g][h].\end{equation} 

Let $D=\seq{d_x}{x\in\power_\kappa(\lambda)^{V[G][H]}}$ be a $\kappa$-slender list in $V[G][H]$. We want to show that there is an ineffable branch $b$ through $D$ in $V[G][H]$.

Let us consider the image of $D$ under $j$: $$j(\seq{d_x}{x\in\power_\kappa(\lambda)^{V[G][H]}})=\seq{d'_y}{y\in\power_{j(\kappa)}(j(\lambda))^{M[G][H][g][h]}}.$$ The set $j''\lambda$ is a subset of $j(\lambda)$ of size $<\! j(\kappa)$. Noting that $j''\lambda \in M$, we define $b \subseteq \lambda \in 
M[G][H][g][h]$ by letting $b = \{\alpha < \lambda \mid j(\alpha) 
\in d'_{j''\lambda}\}$. It remains to show that $b \in V[G][H]$ and is an ineffable 
branch through $D$.

Before we prove the following claim, let us state some simple properties of the lifted embedding $j$. Since we assume $\lambda^{<\lambda}=\lambda$, $\vert H(\lambda)\vert =\lambda$, and this still holds in $V[G][H]$, i.e.\ $\vert H(\lambda)^{V[G][H]}\vert =\lambda$. In (\ref{eq:lift}) we lifted an embedding generated by a normal ultrafilter $U$ on $\power_\kappa(\lambda)$ and therefore $j''\lambda$ is an element of $j(C)$ for every club $C$ in $\power_\kappa(\lambda)^{V[G][H]}$; since there is a bijection between $\lambda$ and $H(\lambda)^{V[G][H]}$, there is a correspondence between $\power_\kappa\lambda^{V[G][H]}$ and $\power_\kappa H(\lambda)^{V[G][H]}$, and in particular it holds for every club $C$ in $V[G][H]$ in $\power_\kappa H(\lambda)^{V[G][H]}$ that \begin{equation}\label{eq:H} j''H(\lambda)^{V[G][H]}\in j(C).\end{equation}

\begin{claim}\label{claim:rest}
For each $x\in\power_\kappa\lambda^{V[G][H]}$, $b \cap x\in M[G][H]$.
\end{claim}

\begin{proof}
By slenderness of $D$, we can fix a club $C$ in $\power_\kappa H(\lambda)^{V[G][H]}$ such that for every $N \in C$, $N \el H(\lambda)^{V[G][H]}$, and for every $x \in N$ of size $\omega_1$, we have $d_{N \cap \lambda} \cap x \in N$.

Let $x\in\power_\kappa\lambda^{V[G][H]}$ be arbitrary. By the $\kappa$-cc of the whole forcing, it holds that $$x\in H(\lambda)^{V[G][H]}=H(\lambda)^{M[G][H]}$$ and therefore $j''x=j(x)\in j'' H(\lambda)^{M[G][H]}$. Let us denote $H(\lambda)^{M[G][H]}$ by $N$. Notice that $j''N$ is an elementary submodel of $j(H(\lambda)^{V[G][H]})$ by a general model-theoretic argument, or by invoking (\ref{eq:H}). Since $j'' H(\lambda)^{M[G][H]}\in j(C)$, we have 
\[
  d'_{(j''N)\cap j(\lambda)} \cap j''x=d'_{j''\lambda} \cap j''x\in j''N.
\] 
It follows that there is $y \subseteq x$ in $N \sub M[G][H]$ such that 
$j(y) = j''y = d'_{j''\lambda} \cap j''x$.
By the definition of $b$, we have $y = b \cap x$, and the proof is finished.
\end{proof}

We can now argue that $b$ is a \emph{cofinal} branch through $D$. To this 
end, fix $x\in\power_\kappa\lambda^{V[G][H]}$. Then, in 
$M[G][H][g][h]$, there is $z \in \power_{j(\kappa)}(j(\lambda))^{M[G][H][g][h]}$ 
such that $z \supseteq j(x)$ and $j(b \cap x) = d'_z \cap j(x)$, namely $z = 
j''\lambda$. By elementarity, there is $z \in \power_\kappa\lambda^{V[G][H]}$ 
such that $z \supseteq x$ and $b \cap x = d_z \cap x$, so $b$ is indeed a cofinal 
branch.

\begin{claim}
$b$ is in $M[G][H]$.
\end{claim}

\begin{proof}
We show that $b$ -- which is approximated in $M[G][H]$ on sets in $\power_\kappa\lambda^{V[G][H]} = \power_\kappa\lambda^{M[G][H]}$ by Claim \ref{claim:rest} -- cannot be added by $Q_\K\times Q_\M$ over $M[G][H]$ and therefore it is already in $M[G][H]$. 

We first argue that $Q_\K$ is $\omega_1$-distributive and $\omega_2$-Knaster over $M[G][H]$. In $M$, $j(\K)\simeq \K*Q_\K$ is $\omega_1$-closed and therefore by the product analysis of $\M$, it is $\omega_1$-distributive in $M[G]$ and therefore $Q_\K$ is $\omega_1$-distributive in $M[G][H]$. Regarding the $\omega_2$-Knaster property, first note that $Q_\K$ is $\kappa$-Knaster in $M[H]$ since $Q_\K$ is forcing equivalent to $\K(\omega_1, j(\kappa)\setminus\kappa)$ and this forcing is $\kappa$-Knaster by a standard $\Delta$-system argument (for more details see \cite{C:trees}). Since $\M$ is $\kappa$-Knaster in $M$ it is still $\kappa$-Knaster in $M[H]$ and therefore $Q_\K$ is $\kappa$-Knaster in $M[H][G]=M[G][H]$. Since $Q_\K$ is $\kappa$-Knaster in $M[G][H]$ it has the $\kappa=\omega_2$-approximation property and cannot add $b$ over $M[G][H]$.

Now, we will show that over $M[G][H][h]$, $Q_\M$ also cannot add $b$. In $M[G][H][h]$, $\kappa=\aleph_2=2^\omega$ and $Q_\M$ is forcing equivalent to $\Add(\omega,\kappa^\dagger)*\M^\kappa$, where $\kappa^\dagger$ is the first inaccessible above $\kappa$. The branch $b$ cannot be added by $\Add(\omega,\kappa^\dagger)$ since this forcing is $\omega_1$-Knaster. 

In $M[G][H][h][\Add(\omega,\kappa^\dagger)]$, $2^\omega = \kappa^\dagger > (\omega_2)^{V[G][H]}$. Moreover, in $V[G][H]$, for all $x \in \power_\kappa\lambda$, we 
have $\vert\{d_z \cap x \mid z \in \power_\kappa\lambda\}\vert 
\leq \vert \power(x) \vert \leq 2^{\omega_1}$, and hence $D$ is a list with width 
at most $((2^{\omega_1})^+)^{V[G][H]}) = (\omega_3)^{V[G][H]}$. By standard product analysis (see Section \ref{mitchell_section}) $\M^\kappa$ is a projection of a product of the form $\Add(\omega,j(\kappa))\times \Q_\kappa$, where $\Q_\kappa$ is $\omega_1$-closed. Again the branch $b$ cannot be added by $\Add(\omega,j(\kappa))$ since this forcing is $\omega_1$-Knaster. Since $D$ has width $(\omega_3)^{V[G][H]}<2^\omega$ in $M[G][H][h][\Add(\omega,\kappa^\dagger)]$, we can apply Lemma \ref{closed_preservation_lemma} to $\Add(\omega,j(\kappa))$ as $\P$ and $\Q_\kappa$ as $\Q$ over the model $M[G][H][h][\Add(\omega,\kappa^\dagger)]$. Therefore the forcing $\Q_\kappa$ cannot add $b$ over the model $M[G][H][h][\Add(\omega,\kappa^\dagger)][\Add(\omega,j(\kappa))]$, hence the cofinal branch $b$ cannot be added by $\M^\kappa$ over $M[G][H][h][\Add(\omega,\kappa^\dagger)]$.
\end{proof}

\begin{claim}
The function $b$ is an ineffable branch in $D$.
\end{claim}

\begin{proof}
We need to show that the set $S=\set{x\in\power_\kappa(\lambda)}{b\cap x= d_x}$ is stationary, hence it is enough to show that $j''\lambda$ is in $j(S)=\set{y\in\power_{j(\kappa)}(j(\lambda))}{j(b) \cap y= d'_y}$. However this follows from the definition of $b$: $j(b)\cap j''\lambda=j''b=d'_{j''\lambda}$.
\end{proof}
This finishes the proof of Theorem \ref{th:SPconverse}.
\end{proof}

In \cite[Question 2]{viale_guessing_models}, Viale asks whether it is consistent 
that there exist $\mu$-guessing models that are not $\omega_1$-guessing for some 
$\mu > \omega_1$. A positive answer is given by Hachtman and Sinapova in 
\cite[Corollary 4.6]{hachtman_sinapova_super_tree_property}; they prove that, 
if $\mu$ is a singular limit of $\omega$-many supercompact cardinals, then, for 
every sufficiently large regular cardinal $\theta$, there are stationarily 
many $\mu^+$-guessing models in $\power_{\mu^{+}}H(\theta)$ that are not 
$\delta$-guessing for any $\delta \leq \mu$. Theorem \ref{th:SPconverse} 
gives a different path to a positive answer, at a smaller cardinal and from 
weaker large cardinal assumptions.

\begin{corollary}
  Let $\kappa$ be a supercompact cardinal. There is a generic extension in which, 
  for all regular $\theta \geq \omega_2$, there are stationarily many 
  $\omega_2$-guessing models $M \in \power_{\omega_2} H(\theta)$ that are not 
  $\omega_1$-guessing.
\end{corollary}

\begin{proof}
  Let $W$ be the generic extension witnessing the conclusion of Theorem 
  \ref{th:SPconverse}.
  Since $\ISP(\omega_2, \omega_2, \lambda)$ holds for all $\lambda \geq 
  \omega_2$, Corollary \ref{isp_guessing_cor} implies that there are stationarily 
  many $\omega_2$-guessing models in $\power_{\omega_2} H(\theta)$ for all 
  regular $\theta \geq \omega_2$. On the other 
  hand, there is an $\omega_1$-Kurepa tree in $W$, so Theorem \ref{th:SP} 
  implies that, for all regular $\theta \geq \omega_2$, the set of 
  $\omega_1$-guessing models is nonstationary in $\power_{\omega_2} H(\theta)$.
  The conclusion follows.
\end{proof}

\bibliographystyle{plain}
\bibliography{bib}

\begin{thebibliography}{10}

\bibitem{abraham}
Uri Abraham.
\newblock Aronszajn trees on {$\aleph \sb{2}$} and {$\aleph \sb{3}$}.
\newblock {\em Ann. Pure Appl. Logic}, 24(3):213--230, 1983.

\bibitem{cox_krueger_indestructible}
Sean Cox and John Krueger.
\newblock Indestructible guessing models and the continuum.
\newblock {\em Fund. Math.}, 239(3):221--258, 2017.

\bibitem{C:trees}
James Cummings.
\newblock Aronszajn and {K}urepa trees.
\newblock {\em Arch. Math. Logic}, 57(1-2):83--90, 2018.

\bibitem{8fold_way}
James Cummings, Sy-David Friedman, Menachem Magidor, Assaf Rinot, and Dima
  Sinapova.
\newblock The eightfold way.
\newblock {\em J. Symb. Log.}, 83(1):349--371, 2018.

\bibitem{D:trees}
Keith~J. Devlin.
\newblock {$\aleph _{1}$}-trees.
\newblock {\em Ann. Math. Logic}, 13(3):267--330, 1978.

\bibitem{fontanella}
Laura Fontanella.
\newblock Strong tree properties for two successive cardinals.
\newblock {\em Arch. Math. Logic}, 51(5-6):601--620, 2012.

\bibitem{fontanella_matet}
Laura Fontanella and Pierre Matet.
\newblock Fragments of strong compactness, families of partitions and ideal
  extensions.
\newblock {\em Fund. Math.}, 234(2):171--190, 2016.

\bibitem{hachtman_sinapova_super_tree_property}
Sherwood Hachtman and Dima Sinapova.
\newblock The super tree property at the successor of a singular.
\newblock {\em Israel J. Math.}, 236(1):473--500, 2020.

\bibitem{hamkins_approximation}
Joel~David Hamkins.
\newblock Extensions with the approximation and cover properties have no new
  large cardinals.
\newblock {\em Fund. Math.}, 180(3):257--277, 2003.

\bibitem{holy_lucke_njegomir}
Peter Holy, Philipp L\"{u}cke, and Ana Njegomir.
\newblock Small embedding characterizations for large cardinals.
\newblock {\em Ann. Pure Appl. Logic}, 170(2):251--271, 2019.

\bibitem{HS:tp}
Radek Honzik and {\v S}{\'a}rka Stejskalov{\'a}.
\newblock The tree property and the continuum function below $\aleph_\omega$.
\newblock {\em Mathematical Logic Quarterly}, 64(1-2):89--102, 2018.

\bibitem{jech_combinatorial_problems}
Thomas Jech.
\newblock Some combinatorial problems concerning uncountable cardinals.
\newblock {\em Ann. Math. Logic}, 5:165--198, 1972/73.

\bibitem{jech}
Thomas Jech.
\newblock {\em Set theory}.
\newblock Springer Monographs in Mathematics. Springer-Verlag, Berlin, 2003.
\newblock The third millennium edition, revised and expanded.

\bibitem{stacking_mice}
Ronald Jensen, Ernest Schimmerling, Ralf Schindler, and John Steel.
\newblock Stacking mice.
\newblock {\em J. Symbolic Logic}, 74(1):315--335, 2009.

\bibitem{krueger_sch}
John Krueger.
\newblock Guessing models imply the singular cardinal hypothesis.
\newblock {\em Proc. Amer. Math. Soc.}, 147(12):5427--5434, 2019.

\bibitem{Kurepa}
Djuro Kurepa.
\newblock Ensembles ordonn\'es et ramifi\'es de points.
\newblock {\em Math. Balkanica}, 7:201--204, 1977.

\bibitem{clh_covering}
Chris Lambie-Hanson.
\newblock Squares and covering matrices.
\newblock {\em Ann. Pure Appl. Logic}, 165(2):673--694, 2014.

\bibitem{narrow_systems}
Chris Lambie-Hanson.
\newblock Squares and narrow systems.
\newblock {\em J. Symb. Log.}, 82(3):834--859, 2017.

\bibitem{lh_lucke}
Chris Lambie-Hanson and Philipp L\"{u}cke.
\newblock Squares, ascent paths, and chain conditions.
\newblock {\em J. Symb. Log.}, 83(4):1512--1538, 2018.

\bibitem{arithmetic_paper}
Chris Lambie-Hanson and \v{S}\'{a}rka Stejskalov\'{a}.
\newblock Strong tree properties and cardinal arithmetic.
\newblock 2022.
\newblock In preparation.

\bibitem{lucke}
Philipp L\"{u}cke.
\newblock Ascending paths and forcings that specialize higher {A}ronszajn
  trees.
\newblock {\em Fund. Math.}, 239(1):51--84, 2017.

\bibitem{magidor_combinatorial_characterization}
Menachem Magidor.
\newblock Combinatorial characterization of supercompact cardinals.
\newblock {\em Proc. Amer. Math. Soc.}, 42:279--285, 1974.

\bibitem{menas}
Telis~K. Menas.
\newblock On strong compactness and supercompactness.
\newblock {\em Ann. Math. Logic}, 7:327--359, 1974/75.

\bibitem{mitchell_approximation}
William~J. Mitchell.
\newblock On the {H}amkins approximation property.
\newblock {\em Ann. Pure Appl. Logic}, 144(1-3):126--129, 2006.

\bibitem{T:wKH}
Stevo~B. Todor\v{c}evi\'{c}.
\newblock Some consequences of {${\rm MA}+\neg{\rm wKH}$}.
\newblock {\em Topology Appl.}, 12(2):187--202, 1981.

\bibitem{viale_guessing_models}
Matteo Viale.
\newblock Guessing models and generalized {L}aver diamond.
\newblock {\em Ann. Pure Appl. Logic}, 163(11):1660--1678, 2012.

\bibitem{viale_weiss}
Matteo Viale and Christoph Wei\ss.
\newblock On the consistency strength of the proper forcing axiom.
\newblock {\em Adv. Math.}, 228(5):2672--2687, 2011.

\bibitem{weiss}
Christoph Wei\ss.
\newblock The combinatorial essence of supercompactness.
\newblock {\em Ann. Pure Appl. Logic}, 163(11):1710--1717, 2012.

\end{thebibliography}

\end{document}